\documentclass[12pt,letterpaper,reqno]{amsart}

\usepackage[margin=1in]{geometry}
\usepackage{graphicx}
\usepackage{amsmath,amsfonts,amssymb,mathtools}
\usepackage{xcolor}
\usepackage{url}

\newtheorem{theorem}{Theorem}[section]
\newtheorem{lemma}[theorem]{Lemma}
\newtheorem{corollary}[theorem]{Corollary}
\newtheorem{proposition}[theorem]{Proposition}
\newtheorem*{proposition*}{Proposition}
\theoremstyle{definition}

\newtheorem*{remark*}{Remark}
 
\mathtoolsset{showonlyrefs=true}

\numberwithin{equation}{section}

\newcommand{\R}{\mathbb{R}}
\newcommand{\Z}{\mathbb{Z}}
\newcommand{\Q}{\mathbb{Q}}
\newcommand{\C}{\mathbb{C}}
\renewcommand{\H}{\mathbb{H}}

\newcommand{\vprod}[1]{\left\langle #1 \right\rangle}
\renewcommand{\Re}{\operatorname{Re}}
\renewcommand{\Im}{\operatorname{Im}}
\newcommand{\abs}[1]{\left\lvert #1 \right\rvert}
\newcommand{\norm}[1]{\left\lVert #1 \right\rVert}
\newcommand{\lrp}[1]{\left(#1\right)}

\newcommand{\lrc}[1]{\left\{#1\right\}}

\newcommand{\ptmatrix}[4]{\left( \begin{smallmatrix} #1 & #2 \\ #3 & #4 \end{smallmatrix} \right)}
\newcommand{\ptMatrix}[4]{\left( \begin{matrix} #1 & #2 \\ #3 & #4 \end{matrix} \right)}
\newcommand{\SLtZ}{\operatorname{SL}_2(\Z)}

\newcommand{\sgn}{\operatorname{sgn}}

\DeclareMathOperator*{\res}{Res}

\newcommand{\V}{\mathcal V}

\DeclareMathOperator{\SL}{SL}
\DeclareMathOperator{\Mp}{Mp}

\DeclareMathOperator{\im}{Im}
\DeclareMathOperator{\re}{Re}
\DeclareMathOperator{\tr}{tr}

\newcommand{\pfrac}[2]{\left(\frac {#1}{#2}\right)}
\renewcommand{\bar}{\overline}
\newcommand{\ep}{\varepsilon}

\newcommand{\e}{\mathfrak{e}}

\title[Formulas for coefficients of mock theta functions]{Transcendental formulas for the coefficients of Ramanujan's mock theta functions}
\author{Nickolas Andersen and Gradin Anderson}
\date{\today}
 
\begin{document}

\begin{abstract}
Ramanujan's 1920 last letter to Hardy contains seventeen examples of mock theta functions which he organized into three ``orders.'' 
The most famous of these is the third-order function $f(q)$ which has received the most attention of any individual mock theta function in the intervening century.
In 1964, Andrews---improving on a result of Dragonette---gave an asymptotic formula for the coefficients of $f(q)$ and conjectured an exact formula for the coefficients: a conditionally convergent series that closely resembles the Hardy--Ramanujan--Rademacher formula for the partition function.
To prove that the conjectured series converges, it is necessary to carefully measure the cancellation among the Kloosterman sums appearing in the formula.
Andrews' conjecture was proved four decades later by Bringmann and Ono. Here we prove exact formulas for all seventeen of the mock theta functions appearing in Ramanujan's last letter.
Along the way, we prove a general theorem bounding sums of Kloosterman sums for the Weil representation attached to a lattice of odd rank.
\end{abstract}

\maketitle

\section{Introduction}

One of the early triumphs of the Hardy-Ramanujan circle method is the famous formula of Rademacher \cite{rademacher}
\begin{equation} \label{eq:hrr}
    p(n) = \frac{2\pi}{(24n-1)^{3/4}} \sum_{c=1}^\infty \frac{A_c(n)}{c} I_{3/2} \pfrac{\pi \sqrt{24n-1}}{6c}
\end{equation}
for the partition function $p(n)$.
Here
\begin{equation} \label{eq:i3/2-def}
    I_{3/2}(x) = \sqrt{\frac{2}{\pi x}} \left(\cosh x - x^{-1}\sinh x\right)
\end{equation}
is the $I$-Bessel function and $A_c(n)$ is the Kloosterman sum
\[
    A_c(n) = \sum_{\substack{d\bmod c \\ (c,d)=1}} e^{\pi i s(d,c)} e\pfrac{-dn}{c},
\]
where $s(d,c)$ is the Dedekind sum.
Since the terms of the sum \eqref{eq:hrr} are given as values of transcendental functions, we refer to Rademacher's formula, and others like it, as transcendental formulas.
By contrast, an algebraic formula for $p(n)$ is given in \cite{bruinier-ono}.
By \eqref{eq:i3/2-def} we have the upper bound $I_{3/2}(\alpha/c) \ll c^{-3/2}$, and combining this with the trivial upper bound $|A_c(n)| \leq c$ we immediately see that the series \eqref{eq:hrr} converges absolutely.

Shortly after Rademacher proved \eqref{eq:hrr}, Lehmer \cite{lehmer} made an extensive study of the Kloosterman sums $A_c(n)$, giving an explicit evaluation for $c$ a prime power, together with a twisted multiplicativity relation in $c$.
This provides a reasonably fast method for evaluating $A_c(n)$ that was recently implemented by Johansson \cite{Johansson} to compute many values of $p(n)$ for large $n$ up to $n=10^{20}$.
Lehmer's results also prove the Weil-type bound
\begin{equation} \label{eq:weil-bound-intro}
    |A_c(n)| \leq 2^{\omega_0(c)}\sqrt{c},
\end{equation}
where $\omega_0(c)$ is the number of distinct odd primes dividing $c$.
Some years after Lehmer's study of $A_c(n)$, Whiteman \cite{whiteman} proved the formula (which he attributes to Selberg)
\begin{equation} \label{eq:whiteman}
    A_c(n) = \sqrt{\frac c{3}} \sum_{\substack{b\bmod {6c} \\b^2 \equiv 1-24n(24c)}} \pfrac{12}{b} \cos\pfrac{\pi b}{6c}.
\end{equation}
Since the sum in \eqref{eq:whiteman} is sparse, this expression provides  another fast method for computing $A_c(n)$ and an immediate proof that $A_c(n) \ll c^{1/2+\ep}$.

The partition function famously appears as the Fourier coefficients of the weight $-1/2$ modular form $\eta^{-1}$, where $\eta$ is the Dedekind eta function.
An alternative expression for the generating function of $p(n)$ is the Eulerian series
\[
    \sum_{n=0}^\infty p(n) q^n = 1 + \sum_{n=1}^\infty \frac{q^{n^2}}{(1-q)^2(1-q^2)^2\cdots (1-q^n)^2}.
\]
This comes from separating the Ferrers diagram of a partition into its Durfee square of size $n^2$ and two partitions with parts of size $\leq n$.
A slight modification of this series yields Ramanujan's third order mock theta function 
\[
    f(q) = 1 + \sum_{n=1}^\infty \frac{q^{n^2}}{(1+q)^2(1+q^2)^2\cdots (1+q^n)^2},
\]
which first appeared in Ramanujan's last letter to Hardy in 1920.
The function $f(q)$ is not a modular form, but there is a nonholomorphic function $\mathcal N$ such that the function $e\pfrac{-\tau}{24}f(e^{2\pi i \tau}) + \mathcal N(\tau)$ transforms like a modular form of weight $1/2$ on $\Gamma_0(2)$; this completed function is an example of a harmonic Maass form.
The $n$-th Fourier coefficient $\alpha_f(n)$ of $f(q)$ counts the number of partitions of $n$ with even rank minus the number with odd rank.

The problem of estimating the size of $\alpha_f(n)$ begins with the claimed first-order asymptotic
\begin{equation} \label{eq:alpha-first-order}
    \alpha_f(n) \sim (-1)^{n-1} \sqrt 6 \frac{e^{\frac\pi{12}\sqrt{24n-1}}}{\sqrt{24n-1}}
\end{equation}
from Ramanujan's last letter. 
This was first proved in 1951 by Dragonette \cite{dragonette}, who also extended \eqref{eq:alpha-first-order} to include $\approx\sqrt n$ terms of exponential size with an error of $O(\sqrt n \log n)$.
Andrews \cite{andrews} further improved this in 1964 by giving an asymptotic series with error $O(n^\ep)$; in fact, he conjectured that the asymptotic series converges conditionally to $\alpha_f(n)$, i.e.
\begin{equation} \label{eq:andrews-dragonette}
    \alpha_f(n) = \frac{\pi}{\sqrt[4]{24n-1}} \sum_{c=1}^\infty \frac{A_{2c}^*(n)}{c} I_{1/2}\pfrac{\pi \sqrt{24n-1}}{12c},
\end{equation}
where
\begin{equation}
    A_{2c}^*(n) = 
    \begin{cases}
        (-1)^{\frac{c+1}{2}} A_{2c}(n) & \text{ if $c$ is odd}, \\
        (-1)^{\frac{c}{2}} A_{2c}(n - \tfrac c2) & \text{ if $c$ is even}.
    \end{cases}
\end{equation}
Finally, in 2006 Bringmann and Ono \cite{BOF} proved that the formula \eqref{eq:andrews-dragonette} holds, and in 2019 Ahlgren and Dunn \cite{AD} proved that the series does not converge absolutely.
An algebraic formula for $\alpha_f(n)$ is given in \cite{bruinier-schwagenscheidt}.

There is an alternative expression for $A_c^*(n)$ that more closely resembles \eqref{eq:whiteman} and elucidates the close symmetry between the formulas \eqref{eq:hrr} and \eqref{eq:andrews-dragonette}.
Define
\begin{equation}
    A_c(n|f) = \sqrt{\frac{c}{3}} \sum_{\substack{b\bmod 6c \\ b^2 \equiv 1-24n(24c)}} \pfrac{-12}{b} \sin\pfrac{-\pi b}{6c}.
\end{equation}
By replacing $b$ by $b+6c$, we see that $A_c(n|f) = 0$ if $c$ is odd.
On the other hand, we have the relation $A_{2c}(n|f) = A_{2c}^*(n)$,
which can be obtained by replacing $b$ by $b(1-6(-1)^{\lfloor (c+1)/2\rfloor}c)$ in the sum defining $A_{2c}(n|f)$. Thus we have
\begin{equation} \label{eq:alpha-f-formula}
    \alpha_f(n) = \frac{2\pi}{\sqrt[4]{24n-1}} \sum_{c=1}^\infty \frac{A_c(n|f)}{c} I_{1/2}\pfrac{\pi \sqrt{24n-1}}{6c}.
\end{equation}

Since $I_{1/2}(x) = (\frac2{\pi x})^{1/2}\sinh x$, the trivial bound $|A_c(n)|\leq c$ is not sufficient to prove convergence of \eqref{eq:andrews-dragonette}.
In fact, the Weil bound \eqref{eq:weil-bound-intro} is also not sufficient. 
Instead, it is necessary to take into account the sign changes of $A_c(n)$ as $c$ varies; this was done in \cite{BOF} by modifying an argument of Hooley \cite{hooley}, and the resulting tail bound was improved on in \cite{AD} using the spectral theory of automorphic forms.

Ramanujan's last letter to Hardy contains a total of seventeen examples of mock theta functions, including $f(q)$.
To each function he assigned an ``order,'' either 3, 5, or 7.
The main purpose of this paper is to give formulas like \eqref{eq:hrr} and \eqref{eq:alpha-f-formula} for the coefficients of each of Ramanujan's mock theta functions.
If $\vartheta$ denotes one of the mock theta functions, let $\alpha_\vartheta(n)$ denote the coefficient of $q^n$ in $\vartheta(q)$.

The mock theta functions of order 3 are
\begin{alignat*}{2}
	f(q) &= \sum_{n=0}^\infty \frac{q^{n^2}}{(-q;q)^2_n}, \quad & \quad \phi(q) &= \sum_{n=0}^\infty \frac{q^{n^2}}{(-q^2;q^2)_n}, \\
	\psi(q) &= \sum_{n=1}^\infty \frac{q^{n^2}}{(q;q^2)_n}, & \chi(q) &= \sum_{n=0}^\infty \frac{q^{n^2}(-q;q)_n}{(-q^3;q^3)_{n}}. 
\end{alignat*}
Here we have used the standard $q$-Pochhammer notation $(a;q)_n := \prod_{m=0}^{n-1} (1-aq^m)$.
Ramanujan's lost notebook contains the third order mock theta function
\[
    \omega(q) = \sum_{n=0}^\infty \frac{q^{2n(n+1)}}{(q;q^2)_{n+1}^2},
\]
which is naturally paired with $f(q)$ via their transformation laws.
Motivated by a conjecture of Andrews \cite{andrews-partitions-2003},
Garthwaite \cite{Garthwaite} proved a transcendental formula for the coefficients of $\omega(q)$ using methods similar to those in \cite{BOF}. 
Our first theorem gives formulas for these five third order mock theta functions.

\begin{theorem} \label{thm:third}
    Let $\vartheta \in \{f, \omega, \chi, \phi, \psi\}$, and let $N$, $d_n$, and $k_n$ be given by the following table.
    \[
        \renewcommand{\arraystretch}{1.2}
        \begin{tabular}{|c|c|c|c|c|c|}
            \hline
             & $f$ & $\omega$ & $\chi$ & $\phi$ & $\psi$  \\
             \hline
             $N$ & $6$ & $6$ & $18$ & $24$ & $24$   \\
             \hline
            $d_n$ & $24n-1$ & $12n+8$ &  $24n-1$ & $24n-1$ & $96n-4$   \\
            \hline
            $k_n$ & $1$ & $-\tfrac 12(-1)^n$ & $1$ & $(-1)^n$ & $1$  \\ \hline
        \end{tabular}
    \]
    Then
    \[
        \alpha_{\vartheta}(n) = \frac{2\pi}{\sqrt[4]{d_n}} \sum_{c>0} \frac{A_c(n|\vartheta)}{c} I_{1/2}\pfrac{\pi \sqrt{d_n}}{Nc},
    \]
    where, if $\vartheta \neq \psi$,
    \begin{equation} \label{eq:third-order-theorem-kloo-sum}
        A_c(n|\vartheta) = k_n \sqrt{\frac{2c}{N}} \sum_{\substack{b \bmod Nc \\ b^2 \equiv -d_n (4Nc)}} \pfrac{-3}{b} \sin\pfrac{-\pi b}{Nc}.
    \end{equation}
    The sum $A_c(n|\psi)$ is given by \eqref{eq:third-order-theorem-kloo-sum} with the added condition that $b\equiv 2, 14, 34, 46 \pmod{48}$.
\end{theorem}

The ten mock theta functions of order 5 are
\begin{alignat*}{2}
	f_0(q) &= \sum_{n=0}^\infty \frac{q^{n^2}}{(-q;q)_n}, \quad & \quad f_1(q) &= \sum_{n=0}^\infty \frac{q^{n(n+1)}}{(-q;q)_n}, \\
	F_0(q) &= \sum_{n=0}^\infty \frac{q^{2n^2}}{(q;q^2)_n}, & F_1(q) &= \sum_{n=0}^\infty \frac{q^{2n(n+1)}}{(q;q^2)_{n+1}}, \\
	\phi_0(q) &= \sum_{n=0}^\infty q^{n^2}(-q;q^2)_n, \qquad & \qquad \phi_1(q) &= \sum_{n=0}^\infty q^{(n+1)^2}(-q;q^2)_n, \\
	\psi_0(q) &= \sum_{n=0}^\infty q^{\frac12(n+1)(n+2)}(-q;q)_n, \qquad & \qquad \psi_1(q) &=  \sum_{n=0}^\infty q^{\frac12n(n+1)} (-q;q)_n, \\
    \chi_0(q) &= \sum_{n=0}^\infty \frac{q^{n}}{(q^{n+1};q)_n}, \qquad & \qquad \chi_1(q) &= \sum_{n=0}^\infty \frac{q^{n}}{(q^{n+1};q)_{n+1}}.
\end{alignat*}
Ramanujan used the same notation for each of the two families of fifth order functions, so we include subscripts, following the convention established by Watson.
The functions $\chi_0$ and $\chi_1$ can be obtained via simple combinations of some of the others, namely
\begin{align} \label{eq:chi0-F0-phi0}
    \chi_0(q) &= 2F_0(q) - \phi_0(-q), \\ \label{eq:chi1-F1-phi1}
    \chi_1(q) &= 2F_1(q) + q^{-1}\phi_1(-q).
\end{align}

Our second theorem gives transcendental formulas for the coefficients of four of the fifth order mock theta functions.
See Theorem~\ref{thm:fifth-2} below for the $f_0(q)$, $f_1(q)$, $F_0(q)$, and $F_1(q)$ formulas.
Then the $\chi_0(q)$ and $\chi_1(q)$ formulas can be obtained immediately by using \eqref{eq:chi0-F0-phi0} and \eqref{eq:chi1-F1-phi1}.

\begin{theorem} \label{thm:fifth}
    Let $\vartheta \in \{\psi_0, \psi_1, \phi_0, \phi_1\}$ and let $d_n$ and $k_n$ be given by the following table.
    \[
    \renewcommand{\arraystretch}{1.2}
        \begin{tabular}{|c|c|c|c|c|}
            \hline
             & $\psi_0$ & $\psi_1$ & $\phi_0$ & $\phi_1$ \\
             \hline
            $d_n$ & $240n-4$ & $240n+44$ & $120n-1$ & $120n-49$  \\
            \hline
            $k_n$ & $-i$ & $1$ & $(-1)^{n}$ & $-i(-1)^n$ \\ \hline
        \end{tabular}
    \]
    Then
    \begin{equation} \label{eq:phi0-formula}
        \alpha_\vartheta(n) = \frac{2\pi}{\sqrt[4]{d_n}} \sum_{c>0} \frac{A_c(n|\vartheta)}{c} I_{1/2}\pfrac{\pi\sqrt{d_n}}{60c},
    \end{equation}
    where
    \begin{equation}
        A_c(n|\vartheta) = k_n \sqrt{\frac{c}{30}} \sum_{\substack{b\bmod 60c \\ b^2 \equiv -d_n (240c)}} \pfrac{3}{b} \chi_5(2,b) \sin\pfrac{-\pi b}{60c},
    \end{equation}
    and $\chi_5(2,\cdot)$ is the Dirichlet character modulo $5$ that maps $2$ to $i$.
\end{theorem}

\begin{remark*}
    Although it is not obvious from the definition, the sum $A_c(n|\vartheta)$ in Theorem~\ref{thm:fifth} is real-valued.
    To see this, two facts are useful.
    First, $\overline{\chi_5(2,\cdot)} = \pfrac{5}{\cdot}\chi_5(2,\cdot)$.
    Second, the values of $b$ in the sum defining $A_c(n|\vartheta)$ satisfy $b\equiv \pm 1 \pmod{5}$ for $\vartheta \in  \{\psi_0, \phi_1\}$, and $b\equiv \pm 2\pmod{5}$ for $\vartheta \in  \{\psi_1, \phi_0\}$.
    Thus, in each case $\overline{k_n} = \pfrac 5b k_n$.
\end{remark*}

The three functions of order 7 are
\[
	\mathcal{F}_0(q) = \sum_{n=0}^\infty \frac{q^{n^2}}{(q^{n+1};q)_n}, \quad  \quad \mathcal{F}_1(q) = \sum_{n=0}^\infty \frac{q^{(n+1)^2}}{(q^{n+1};q)_{n+1}},  \quad \mathcal{F}_2(q) = \sum_{n=0}^\infty \frac{q^{n(n+1)}}{(q^{n+1};q)_{n+1}}.
\]

\begin{theorem} \label{thm:seventh}
    Let $\vartheta \in \{\mathcal F_0, \mathcal F_1, \mathcal F_2\}$ and let $d_n$ and $k_n$ be given by the following table.
    \[
    \renewcommand{\arraystretch}{1.2}
        \begin{tabular}{|c|c|c|c|}
            \hline
             & $\mathcal F_0$ & $\mathcal F_1$ & $\mathcal F_2$  \\
             \hline
            $d_n$ & $168n-1$ & $168n-25$ & $168n+47$   \\
            \hline
            $k_n$ & $1$ & $-1$ & $-1$ \\ \hline
        \end{tabular}
    \]
    Then
    \[
        \alpha_\vartheta(n) = \frac{2\pi}{\sqrt[4]{d_n}} \sum_{c>0} \frac{A_c(n|\vartheta)}{c} I_{1/2}\pfrac{\pi \sqrt{d_n}}{42c},
    \]
    where
    \[
        A_c(n|\vartheta) = k_n \sqrt{\frac{c}{21}} \sum_{\substack{b\bmod 42c \\ b^2 \equiv -d_n(168c)}} \pfrac{-21}{b} \sin\pfrac{-\pi b}{42c}.
    \]
\end{theorem}

As in the case of $\alpha_f(q)$, the convergence of each of the formulas above is difficult to establish since the sign changes of $A_c(n|\vartheta)$ must be taken into account.
We will show that for each of the mock theta functions, the corresponding Kloosterman sum $A_c(n|\vartheta)$ satisfies
\begin{equation} \label{eq:intro-sum-kloo-sum}
    \sum_{c\leq x} \frac{A_c(n|\vartheta)}{c} \ll x^{1/2-\delta}
\end{equation}
for some $\delta>0$ depending on $\vartheta$.
To obtain the estimate \eqref{eq:intro-sum-kloo-sum}, we apply the Goldfeld--Sarnak method to a general class of Kloosterman sums.
Recent work of the authors and Woodall~\cite{AAW} shows that the sums $A_c(n|\vartheta)$ are linear combinations of Kloosterman sums for the Weil representation attached to a lattice of odd rank (see Section~\ref{sec:weil-rep} for definitions).
Corollary~\ref{cor:gs} below gives a bound of the form \eqref{eq:intro-sum-kloo-sum} for such Kloosterman sums.

In Theorem~\ref{thm:table-formulas} we give a single formula for the coefficients of all seventeen of Ramanujan's original mock theta functions in terms of the following data given in Section~\ref{sec:mock-theta-traces}:
\begin{enumerate}
    \item the level $N \in \Z_+$ that defines the lattice and bilinear form for the Weil representation,
    \item the character $\ep$ of the group $U_N$ of unitary divisors of $N$,
    \item the order $m \in \Z_+$ of the pole of the associated Poincar\'e series on $\Gamma_0(N)$, 
    \item a linear expression $d_n$ of the form $An+B$,
    \item the index $r\in \Z/2N\Z$ that identifies the component of the vector-valued harmonic Maass form where $\vartheta$ lives, and
    \item a constant $c_n\in \C$, possibly depending on the residue class of $n$ modulo $2$, $3$, or $4$.
\end{enumerate}
Each mock theta function appearing in this paper is associated with one of  five $(N,\ep)$ pairs, specifically those in \eqref{eq:N-ep-pairs} below.
Other $(N,\ep)$ pairs give rise to the mock theta functions appearing in Ramanujan's lost notebook and elsewhere (for an excellent survey of mock theta functions, see \cite{Gordon-McIntosh}).
For example, the sixth order function\footnote{A formula for the coefficients of $\gamma$ was proved in \cite{sun}.} $\gamma$ corresponds to $(18,\ep_2)$ in the notation of Section~\ref{sec:alg-traces-poincare-series}.
Formulas for many of these other mock theta functions appear in the second author's Master's thesis \cite{anderson-thesis}.
Additionally, there are several $(N,\ep)$ pairs for which the methods of this paper yield formulas for the coefficients of vector-valued harmonic Maass forms, for example
\[
    (10, \ep_2), \qquad (22, \ep_2), \qquad (46, \ep_2), \qquad (70, \ep_2\ep_5\ep_7).
\]
Computationally, the formula in Theorem~\ref{thm:table-formulas} appears to output integers for these pairs, for some $m,d_n,r,c_n$ inputs, but we have not been able to identify known mock theta functions whose coefficients match these outputs.
It would be interesting to investigate this further.

We conclude this introduction with a brief outline of the paper.
In Section~\ref{sec:alg-traces-poincare-series} we give some background information on the Weil representation and show how each of the mock theta functions can be completed to become components of vector-valued harmonic Maass forms with the Weil representation.
This has been done already for the fifth and seventh order functions, and for some of the third order functions. 
In Appendix~\ref{appendix} we give details for the remaining third order functions.
In Section~\ref{sec:transcendental-formulas} we apply work of Alfes, Bruinier, and Schwagenscheidt on the Millson theta lift to write the coefficients of the mock theta functions as algebraic traces of Poincar\'e series.
At that point, if we ignore issues of convergence then a straightforward argument, detailed in Section~\ref{sec:transcendental-formulas}, yields the transcendental formulas\footnote{This path to the transcendental formulas is somewhat indirect, since it passes through the algebraic formulas along the way. We have chosen to present the proof in this way since it highlights the close connection between the algebraic and transcendental formulas.}.
Finally, the crucial Kloosterman sum estimate that addresses the subtle convergence issues is proved in Section~\ref{sec:goldfeld-sarnak}.

\section{Ramanujan's mock theta functions and the Weil representation}

In the early 2000s, Zwegers \cite{zwegers-f,zwegers-thesis} proved that the mock theta functions can be completed to harmonic Maass forms, and he gave explicit constructions for the third order functions $f(q)$ and $\omega(q)$, and for all of the fifth and seventh order functions.
In \cite{Andersen-Order-5,andersen-7} the first author reinterpreted Zwegers' results for the fifth and seventh order functions in the context of vector-valued harmonic Maass forms with the Weil representation. 
In this section, we recall the relevant definitions and write each of Ramanujan's mock theta functions as components of such vector-valued forms.

\subsection{The Weil representation} \label{sec:weil-rep}

Good references for the background material in this section are \cite[Chapter~1]{bruinier} and \cite[Chapter~14]{cohen-stromberg}.
Let $L$ be an even lattice with nondegenerate symmetric bilinear form $\langle \cdot, \cdot \rangle$, and let $q(x) = \frac 12 \vprod {x,x}$ denote the associated $\Z$-valued quadratic form.
Let $L'$ denote the dual lattice
\begin{equation} \label{eq:L'-def}
    L' = \left\{ x\in L\otimes \Q : \vprod {x,y} \in \Z \text{ for all }y\in L \right\};
\end{equation}
then the quotient $L'/L$ is a finite abelian group.
We denote the standard basis of $\C[L'/L]$ by $\{\mathfrak e_\alpha : \alpha \in L'/L\}$.
By identifying $L$ with $\Z^g$ we may write $\vprod{x,y} = x^TMy$ for all $x,y\in \Z^g$, for some symmetric integer matrix $M$ with even diagonal.
Let $\Delta = \det M$; then $|L'/L|=|\Delta|$.

Let $\Mp_2(\R)$ be the metaplectic group, the elements of which are of the form $(\gamma,\phi)$, where $\gamma=\ptmatrix abcd\in \SL_2(\R)$ and $\phi:\mathbb H\to\C$ is a holomorphic function with $\phi^2(\tau)=c\tau+d$.
The group law on $\Mp_2(\R)$ is given by
\begin{equation}
	(\gamma_1,\phi_1(\tau))(\gamma_2,\phi_2(\tau)) = (\gamma_1\gamma_2, \phi_1(\gamma_2\tau)\phi_2(\tau)).
\end{equation}
Let $\Mp_2(\Z)$ denote the inverse image of $\SL_2(\Z)$ under the covering map $(\gamma,\phi)\mapsto \gamma$.
Then $\Mp_2(\Z)$ is generated by the elements
\begin{equation}
	T=\left(\ptmatrix 1101, 1\right) \quad \text{ and } \quad S = \left(\ptmatrix 0{-1}10,\sqrt\tau\right)
\end{equation}
and the center of $\Mp_2(\Z)$ is generated by
\begin{equation}
	Z = S^2 = (ST)^3 = \left(\ptmatrix {-1}00{-1},i\right).
\end{equation}
The Weil representation associated with the lattice $L$ is the unitary representation
\[
    \rho_L : \Mp_2(\Z) \to \C[L'/L]
\]
given by
\begin{align}
    \rho_L(T) \e_\alpha &= e(q(\alpha))\e_\alpha, \label{eq:T-transform} \\
    \rho_L(S) \e_\alpha &= \frac{i^{(b^--b^+)/2}}{\sqrt{|L'/L|}} \sum_{\beta \in L'/L} e(-\vprod{\alpha,\beta})\e_\beta. \label{eq:S-transform}
\end{align}
Here $(b^+,b^-)$ is the signature of $L$.
For $\mathfrak g\in \Mp_2(\Z)$ we define the coefficient $\rho_{\alpha\beta}(\mathfrak g)$ of the representation $\rho_L$ by
\begin{equation}
    \rho_L(\mathfrak g) \mathfrak e_\beta = \sum_{\alpha\in L'/L} \rho_{\alpha\beta}(\mathfrak g) \mathfrak e_\alpha.
\end{equation}
Shintani \cite{shintani} gave the following formula for the coefficients $\rho_{\alpha\beta}(\mathfrak g)$: if $\mathfrak g = (\ptmatrix abcd, \sqrt{cz+d})$ and $c>0$, then
\begin{equation} \label{eq:rho-alpha-beta-formula}
    \rho_{\alpha\beta}(\mathfrak g) = \frac{i^{\frac{b^--b^+}{2}}}{c^{(b^++b^-)/2}\sqrt{|L'/L|}} \sum_{r\in L/cL} e\pfrac{aq(\alpha+r)-\vprod{\beta,\alpha+r}+dq(\beta)}{c}.
\end{equation}

A harmonic Maass form of weight $k$ and type $\rho_L$ is a real analytic function $f:\H\to \C[L'/L]$ that (1) satisfies the transformation law
\begin{equation} \label{eq:weil-rep-transformation-law}
    f(\gamma \tau) = \phi^{2k}(\tau) \rho_L(\mathfrak g) f(\tau) \quad \text{ for all } \mathfrak g = (\gamma,\phi)\in \Mp_2(\Z),
\end{equation}
(2) is annihilated by the weight $k$ hyperbolic Laplacian
\[
    \Delta_k = -y^2 \left(\frac{\partial^2}{\partial x^2} + \frac{\partial^2}{\partial y^2}\right) + i k y \left( \frac{\partial}{\partial x} + i \frac{\partial}{\partial y} \right),
\]
and (3) can be written in the form $f(\tau) = P(\tau) + g(\tau)$, where $g$ decays exponentially at infinity and
\[
    P(\tau) = \sum_{\alpha\in L'/L} \left(c_{\alpha,0} + \sum_{n=1}^N c_{\alpha,n} e(-m_{\alpha,n}\tau)\right)\mathfrak e_\alpha
\]
for some $m_{\alpha,n}\in \Q_+$.
We call $P(\tau)$ the principal part of $f$.
Let $\mathcal H_{k}(\rho_L)$ denote the vector space of harmonic Maass forms of weight $k$ and type $\rho_L$, and let $M_k^!(\rho_L)$, $M_k(\rho_L)$, and $S_k(\rho_L)$ denote the subspaces of weakly holomorphic, holomorphic, and cuspidal modular forms, respectively.
Setting $\mathfrak g=Z$ in \eqref{eq:weil-rep-transformation-law}, we find that every $f\in \mathcal H_k(\rho_L)$ satisfies $f = (-1)^{2k+b^--b^+} f$.
Thus $f=0$ unless $k$ satisfies the consistency condition
\begin{equation} \label{eq:sigma-consistency}
    \sigma := k + \tfrac 12(b^--b^+) \in \Z.
\end{equation}
In particular, we see that $k\in \frac 12\Z$.

Each harmonic Maass form $f$ can be written as $f^++f^-$, where $f^+$ is holomorphic on $\H$ with a possible pole at infinity and $f^-$ is nonholomorphic but satisfies $\xi_k f^- \in M_{2-k}(\bar \rho_L)$, where $\xi_k=2iy^k \bar\partial_{\bar\tau}$.
The function $f^+$ is called the holomorphic part of $f$.

\subsection{Mock theta functions} \label{sec:mock-theta-vectors}

Here we describe how the completed mock theta functions appear as components of vector-valued harmonic Maass forms with the Weil representation.
In each case the lattice is of the form $L=L(N)=\Z$ with bilinear form $\langle x,y \rangle = -2Nxy$ for some $N\in \Z_+$.
It follows that $\Delta=-2N$ and $L'/L\cong \Z/2N\Z$.
For brevity, when working with $L(N)$ we will write $\mathfrak e_h$ instead of $\mathfrak e_{h/2N}$ when $\frac{h}{2N}\in L'/L$.

Each of the completion terms for Ramanujan's mock theta functions can be written in terms of the unary theta function
\[\theta_{N,a}(\tau)=\sum_{\substack{n\in \Z\\n\equiv a(2N)}}nq^{n^2/4N}\]
and the associated nonholomorphic Eichler integral\footnote{For comparison, this is related to the notation used by Zwegers in Definition 4.1 of \cite{zwegers-thesis} by 
$R_{k/N,1/2}\lrp{\frac{N}{2}\tau}= -e\pfrac{-k}{2N}(\V_{N,k}(\tau)+\V_{N,N-k}(\tau))$.}
\[
    {\V_{N,a}}(\tau) := \frac{i}{\sqrt{2N}} \int_{-\overline{\tau}}^{i\infty} \frac{\theta_{N,a}(z)}{\sqrt{-i(\tau+z)}}dz.
\]
Here $q=e(\tau)=e^{2\pi i \tau}$.

In \cite{zwegers-f}, Zwegers gave the transformation laws for the completed versions of the third order mock theta functions $f(q)$ and $\omega(q)$.
If we define
\begin{align}
    \tilde f(\tau) &= q^{-\frac{1}{24}} f(q) +(2\V_{6,1}-2\V_{6,5})(\tau)\\
    \tilde \omega(\tau) &= 2q^{\frac 23} \omega(q) -(2\V_{6,2}-2\V_{6,4})(2\tau)
\end{align}
then the vector $g(\tau) = (\tilde f(\tau), \tilde \omega(\frac \tau2), \zeta_3^{-1}\tilde \omega(\frac{\tau+1}2))^T$ satisfies the transformations
\begin{equation} \label{eq:third-order-h-T-S}
    g(\tau+1) = 
    \left(
    \begin{array}{ccc}
        \zeta_{24}^{-1} & 0 & 0  \\
        0 & 0 & \zeta_3 \\
        0 & \zeta_3 & 0
    \end{array}
    \right) g(\tau)
    \quad \text{ and } \quad
    g(-1/\tau) = \sqrt{-i\tau}
    \left(
    \begin{array}{ccc}
        0 & 1 & 0  \\
        1 & 0 & 0 \\
        0 & 0 & -1
    \end{array}
    \right) g(\tau),
\end{equation}
where $\zeta_m = e(\frac 1m)$.
Then a calculation involving \eqref{eq:T-transform}, \eqref{eq:S-transform}, and \eqref{eq:third-order-h-T-S} yields the following lemma.

\begin{lemma}
Let
\begin{multline}
    H_{(3,1)}(\tau) = \sum_{h\in \{1,5\}} (\tfrac{-12}{h}) \tilde f(\tau) \mathfrak (\e_h-\e_{-h}) \\ + \big( \zeta_3^{-1}\tilde \omega(\tfrac{\tau+1}2) - \tilde \omega(\tfrac{\tau}2)\big)(\mathfrak e_2 - \mathfrak e_{10}) - \big( \zeta_3^{-1}\tilde \omega(\tfrac{\tau+1}2) + \tilde \omega(\tfrac{\tau}2)\big)(\mathfrak e_4 - \mathfrak e_{8}).
\end{multline}
Then $H_{(3,1)}$ has principal part $q^{-1/24}(\mathfrak e_1 - \mathfrak e_5 + \mathfrak e_7 - \mathfrak e_{11})$ and is an element of $\mathcal H_{1/2}(\rho_{L(6)})$.
\end{lemma}

\begin{remark*}
    Note that the even-indexed coefficients of $\omega(q)$ appear in two of the components of $H_{(3,1)}$, while the odd-indexed coefficients appear in two other components.
\end{remark*}

Transformation laws for the third order mock theta functions $\phi(q)$ and $\psi(q)$ can be found in \cite{Gordon-McIntosh}; they involve an additional mock theta function $\nu(q)$ whose definition and completion $\tilde\nu$ are in Appendix~\ref{appendix}.
Let 
\begin{align*}
    \tilde{\psi}(\tau)&= q^{-\frac1{24}}\psi(q) -\frac12 \sum_{h\in\lrc{2,10,14,22}}\V_{24,h}(\tau),\\
    \tilde{\phi}(\tau) &= q^{-\frac1{24}}\phi(q) + \sum_{h\in\lrc{1,5,7,11}}\V_{24,h}(4\tau)-\sum_{h\in\lrc{13,17,19,23}}\V_{24,h}(4\tau).
\end{align*}
We need to split the coefficients of $\phi(q)$ by their index modulo $4$; to that end, we define
\[
    \tilde{\phi}^{(k)}(\tau) = \frac14\sum_{\ell=0}^3 i^{-k\ell}\zeta_{96}^{\ell} \tilde{\phi}\pfrac{\tau+\ell}{4}.
\]
The following lemma is proved in Appendix~\ref{appendix}.

\begin{lemma} \label{lem:order-3-2}
Let
\begin{multline*}
    {H}_{(3,2)} = \sum_{h\in\{1,31\}} \tilde{\phi}^{(0)} (\e_h - \e_{-h}) -  \sum_{h\in \{13,19\}} \tilde{\phi}^{(1)} (\e_h - \e_{-h}) 
    + \sum_{h\in \{7,25\}} \tilde{\phi}^{(2)}(\e_h - \e_{-h}) \\
    + \sum_{h\in \{5,11\}} \tilde{\phi}^{(3)}(\e_h - \e_{-h})
    - \sum_{h\in \{8,16\}} \tilde{\nu}(\e_h - \e_{-h}) 
    - \sum_{h\in \{2, 10, 14, 22\}} \tilde{\psi} (\e_h - \e_{-h}).
\end{multline*}
Then ${H}_{(3,2)}$ has principal part $q^{-1/96}(\e_1-\e_{17}+\e_{31}-\e_{47})$ and is an element of $\mathcal{H}_{1/2}(\rho_{L(24)}).$ 
\end{lemma}

The situation is similar with the final third order function $\chi(q)$.
Let 
\begin{align*}
    \tilde{\chi}(\tau) &= q^{-\frac1{24}}\chi(q) +\frac12\sum_{\substack{0<h<18\\\gcd(h,18)=1}}\pfrac{-3}{h}\V_{18,h}(3\tau),
\end{align*}
and, to split the coefficients by their index modulo 3, define
\[
    \tilde{\chi}^{(k)}(\tau) = \frac13\sum_{\ell=0}^2 \zeta_3^{-k\ell}\zeta_{72}^{\ell} \tilde{\chi}\pfrac{\tau+\ell}{3}.
\]
The transformation law for $\tilde\chi^{(k)}$ involves three other mock theta functions called $\sigma$, $\xi$, and $\rho$.
Their definitions and completions are in Appendix~\ref{appendix}, together with the proof of the following lemma.

\begin{lemma} \label{lem:order-3-3}
Let
\begin{multline*}
    {H}_{(3,3)} 
    = \tilde{\xi}^{(0)}(-\e_{12}+\e_{24}) + \tilde{\xi}^{(1)}(-\e_6+\e_{30}) + \sum_{h\in \{1,19\}}\tilde{\chi}^{(0)}(\e_h - \e_{-h}) 
    + \sum_{h\in \{7,25\}} \tilde{\chi}^{(1)}(\e_h - \e_{-h}) \\
    - \sum_{h\in \{5,23\}}\tilde{\chi}^{(2)}(\e_h - \e_{-h})
    + \sum_{h\in \{3, 21\}} \tilde{\sigma}(\e_h - \e_{-h})
    + 2 \sum_{k=0}^5 (-1)^{k+1} \tilde \rho^{(k)}(\e_{6k+8} - \e_{-6k-8}). 
\end{multline*}
Then ${H}_{(3,3)}$ has principal part equal to $q^{-1/72}(\e_{1}-\e_{17}+\e_{19}-\e_{35})+(-\e_{12}+\e_{24})$ and is an element of $\mathcal{H}_{1/2}(\rho_{L(18)}).$
\end{lemma}

Building on the results in Sections~4.3 and 4.4 of \cite{zwegers-thesis}, the first author expressed the fifth order mock theta functions\footnote{In view of \eqref{eq:chi0-F0-phi0} and \eqref{eq:chi1-F1-phi1} we will not discuss $\chi_0(q)$ or $\chi_1(q)$ for the remainder of the paper.} as harmonic Maass forms of type $\rho_{L(60)}$ as follows.
Let
\begin{align*}
    \tilde{f}_0(\tau) &= q^{-\tfrac{1}{60}}f_0(q) +\sum_{\substack{{0<h<60}\\{h\equiv\pm2 (10)}\\{\gcd(h,60)=2}}}\V_{60,h}(\tau),\\
    \tilde{f}_1(\tau) &= q^{\tfrac{11}{60}}f_1(q) +\sum_{\substack{{0<h<60}\\{h\equiv\pm4 (10)}\\{\gcd(h,60)=2}}}\V_{60,h}(\tau), \\
    \tilde{F}_0(\tau) &= q^{-\tfrac{1}{120}}(F_0(q)-1) +\frac12 \sum_{\substack{{0<h<60}\\{h\equiv\pm1 (10)}\\{\gcd(h,60)=1}}} a_h\V_{60,h} (2\tau), \\
    \tilde{F}_1(\tau) &= q^{\tfrac{71}{120}}F_1(q) +\frac12 \sum_{\substack{{0<h<60}\\{h\equiv\pm3 (10)}\\{\gcd(h,60)=1}}} a_h\V_{60,h} (2\tau),
\end{align*}
where
\[
    a_h=\begin{cases}1& \text{ if } 0<h<30,\\-1& \text{ otherwise}.\end{cases}
\]

Define $\tilde \phi_0, \tilde \phi_1, \tilde \psi_0, \tilde \psi_1$ by making the replacements $f_0\to -\psi_0$, $f_1 \to -\psi_1$, $F_0(q)-1 \to -\frac 12\phi_0(-q)$, $F_1(q)\to \frac 12q^{-1}\phi_1(-q)$, respectively, above (compare with $F_{5,2}$ and $G_{5,2}$ from Section~4.4 of \cite{zwegers-thesis}).
Then Lemma~5 of \cite{Andersen-Order-5}, together with the discussion in Section~6 of that paper, gives the following.
\begin{lemma} \label{lem-order5}
    Define
    \begin{multline}
        {H}_{(5,1)}(\tau)
        =\sum_{\substack{{0<h<60}\\{h\equiv\pm2 (10)}\\{\gcd(h,60)=2}}}\tilde{f}_0(\tau) (\mathfrak{e}_{h}-\mathfrak{e}_{-h})-\sum_{\substack{{0<h<60}\\{h\equiv\pm1 (10)}\\{\gcd(h,60)=1}}}(a_{h} \tilde{F}_0\lrp{\tfrac{\tau}{2}}+b_{h}  \zeta_{240}\tilde{F}_0(\tfrac{\tau+1}{2})) (\mathfrak{e}_{h}-\mathfrak{e}_{-h})\\
        +\sum_{\substack{0<h<60\\h\equiv\pm4 (10)\\\gcd(h,60)=2}}\tilde{f}_1(\tau) (\mathfrak{e}_{h}-\mathfrak{e}_{-h})-\sum_{\substack{0<h<60\\h\equiv\pm3 (10)\\\gcd(h,60)=1}}(a_{h} \tilde{F}_1(\tfrac{\tau}{2})+b_{h} \zeta_{240}^{-71}\tilde{F}_1(\tfrac{\tau+1}{2})) (\mathfrak{e}_{h}-\mathfrak{e}_{-h}), \label{eq:H51-def}
    \end{multline}
    where
    \begin{equation}
        b_h=\begin{cases}1& \text{ if } h\equiv\pm1,\pm13\pmod{60},\\-1& \text{ otherwise.}\end{cases}
    \end{equation}
    Define ${H}_{(5,2)}$ by making the replacements $\tilde f_0\to \tilde\psi_0$, $\tilde f_1 \to \tilde\psi_1$, $\tilde F_0 \to \tilde \phi_0$, $\tilde F_1\to \tilde\phi_1$ in \eqref{eq:H51-def}.
    Then ${H}_{(5,1)}, {H}_{(5,2)} \in \mathcal H_{1/2}(\rho_{L(60)})$.
    Furthermore, the principal part of $H_{(5,1)}$ equals
    \[
        q^{-\frac{1}{60}}(\e_2+\e_{22}+\e_{38}+\e_{58} - \e_{62}-\e_{82}-\e_{98}-\e_{118}),
    \]
    while the principal part of $H_{(5,2)}$ equals
    \[
        q^{-\frac{1}{240}}(\e_1-\e_{31}-\e_{41}-\e_{49}+\e_{71}+\e_{79}+\e_{89}-\e_{119}).
    \]
\end{lemma}

For the seventh order functions, let
\begin{align}
    \tilde {\mathcal F}_0(\tau) &= q^{-\frac1{168}} \mathcal F_0(q) + \sum_{h\in \{1, 13, 29, 41\}} (\tfrac{-21}{h})\V_{42,h}(\tau)\\
    \tilde {\mathcal F}_1(\tau) &= q^{-\frac{25}{168}} \mathcal F_1(q)- \sum_{h\in \{5, 19, 23, 37\}}\V_{42,h}(\tau)  \\
    \tilde {\mathcal F}_2(\tau) &= q^{\frac{47}{168}} \mathcal F_2(q) - \sum_{h\in \{11, 17, 25, 31\}}\V_{42,h}(\tau).
\end{align}
Then Lemma~4 of \cite{andersen-7} gives the following.

\begin{lemma}
    Define
    \begin{multline}
    {H}_{(7)}(\tau) = \sum_{h\in \{1, 13, 29, 41\}} (\tfrac{-21}{h}) \tilde {\mathcal F_0}(\tau)(\mathfrak e_h - \mathfrak e_{-h}) \\- \sum_{h\in \{5, 19, 23, 37\}} \tilde {\mathcal F_1}(\tau)(\mathfrak e_h - \mathfrak e_{-h}) - \sum_{h\in \{11, 17, 25, 31\}} \tilde {\mathcal F_2}(\tau)(\mathfrak e_h - \mathfrak e_{-h}).
    \end{multline}
    Then the principal part of $H_{(7)}$ equals $q^{-1/168}(\e_1-\e_{13}-\e_{29}+\e_{41}-\e_{43}+\e_{55}+\e_{71}-\e_{83})$ and $H_{(7)} \in \mathcal H_{1/2}(\rho_{L(42)})$.
\end{lemma}

\section{Algebraic traces of Poincar\'e series}

\label{sec:alg-traces-poincare-series}

Bruinier and Schwagenscheidt \cite{bruinier-schwagenscheidt} expressed the coefficients of the third order mock theta functions $f(q)$ and $\omega(q)$ as algebraic traces of CM values of a $\Gamma_0(6)$-invariant modular function.
Following their work, Klein and Kupka \cite{kk} gave similar formulas for each of the fifth and seventh order functions, together with other second, sixth, eighth, and tenth order mock theta functions not appearing in Ramanujan's last letter. 
Using their results as a starting point, we relate all of Ramanujan's mock theta function coefficients to algebraic traces of CM values of Poincar\'e series.
This will quickly lead to the transcendental formulas in Section~\ref{sec:transcendental-formulas}, assuming uniform convergence of the relevant sums in the auxiliary parameter. 
These convergence issues will be settled in Section~\ref{sec:goldfeld-sarnak}.

\subsection{Quadratic forms and Atkin-Lehner involutions}\label{sec:quad-forms-atkin-lehner} 
For a negative discriminant $D$, let $\mathcal Q_{N,D}$ denote the set of integral binary quadratic forms
\[
\mathcal Q_{N,D}= \lrc{[a,b,c]=ax^2+bxy+cy^2: b^2-4ac=D \text{ and } N|a}.
\]
These are either positive or negative definite, and we write $\mathcal Q^{\pm}_{N,D}$ to denote the forms in $\mathcal Q_{N,D}$ which are positive or negative definite, respectively.
For $r\in \Z$ satisfying $r^2\equiv D\pmod{4N}$, let $\mathcal Q_{N,D,r}^{\pm}$ denote the subset
\[
    \mathcal Q_{N,D,r}^{\pm} = \{ [a,b,c] \in \mathcal Q_{N,D}^{\pm} : b\equiv r\pmod{2N} \}.
\]
Matrices in $\SL_2(\R)$ act on quadratic forms via
\begin{equation} \label{eq:sl2r-action-on-qforms}
    \ptmatrix \alpha\beta\gamma\delta Q(x,y) = Q(\delta x - \beta y, -\gamma x+\alpha y).
\end{equation}
In particular, if $g\in \Gamma_0(N)$ and $Q\in \mathcal Q_{N,D,r}^\pm$ then $gQ \in \mathcal Q_{N,D,r}^\pm$.

If $F$ is a $\Gamma_0(N)$-invariant function, then the trace of $F$ of index $D,r$ is defined as
\begin{equation} \label{eq:tr-F-pm-def}
    \tr_F^\pm(D,r) = \sum_{Q\in \Gamma_0(N) \backslash \mathcal Q^\pm_{N,D,r}} \frac{F(\tau_Q)}{\omega_Q},
\end{equation}
where
\begin{equation}
    \tau_Q = \frac{-b}{2a} + i\frac{\sqrt{|D|}}{2|a|}
\end{equation}
is the solution of the equation $Q(\tau,1)=0$ in the upper half-plane and $\omega_Q$ is one-half the order of the stabilizer $\Gamma_0(N)_Q$ of $Q$.
Note that for $g\in \SL_2(\R)$ we have
\begin{equation} \label{eq:zgq=gzq}
    \tau_{gQ} = g \tau_Q,
\end{equation}
so the sum in \eqref{eq:tr-F-pm-def} is well-defined.

The following result of Alfes expresses the Fourier coefficients of the Millson theta lift\footnote{We follow the convention of \cite{bruinier-schwagenscheidt} in multiplying the Millson theta lift of \cite{alfes} by $i/\sqrt{N}$.} of $F$ in terms of the traces $\tr_F$. 

\begin{theorem}[Theorem~4.3.1 of \cite{alfes}]\label{thm:Alfes-Trace}
    Let $F$ be a $\Gamma_0(N)$-invariant modular function.
    Then the Millson theta lift $\mathcal {I}^M(F,\tau) = \sum_{h}g_h \e_h$ is a harmonic Maass form of weight $1/2$ and type $\rho_{L(N)}$.
    Let $D$ be a negative discriminant and choose $r\in \Z/2N\Z$ with $D\equiv r^2\pmod{4N}$.
    Then the coefficient of $q^{|D|/4N}$ in the holomorphic part of $g_r$ equals
    \begin{equation}\label{eq:alfes-trace}
        \frac{i}{\sqrt{|D|}} \left( \tr_F^+(D,r) - \tr_F^-(D,r) \right).
    \end{equation}
\end{theorem}

We are primarily interested in modular functions with Atkin-Lehner symmetries.
There is one Atkin-Lehner involution for each unitary divisor $u$ of $N$, meaning that $(u,N/u)=1$. We will sometimes write this condition as $u\mid\mid N$. The set $U_N$ of unitary divisors of $N$ forms a group under the operation 
\[
    a\ast b = \frac{ab}{(a,b)^2}.
\]
Note that every non-identity element of $U_N$ has order $2$, so any character of $U_N$ must take values in $\{-1,1\}$.
The Atkin-Lehner involution corresponding to $u\in U_N$ is the map $F(\tau) \mapsto F(W_u \tau)$, where
\[
    W_u= \frac{1}{\sqrt{u}}\ptMatrix{\alpha u}{\beta}{\gamma N}{\delta u} \text{ for any } \alpha, \beta, \gamma, \delta \in \Z \text{ with } \alpha\delta u^2 - \beta \gamma N = u.
\]
If $F$ is $\Gamma_0(N)$-invariant then this map does not depend on the choice of $\alpha, \beta, \gamma, \delta$.
If $u$ and $r$ are both unitary divisors of $N$ then $W_u W_{r} = \gamma W_{u\ast r}$ for some $\gamma\in \Gamma_0(N)$.

Let $Q=[a,b,c]$, with $N\mid a$ and $b\equiv r\pmod{2N}$, be a quadratic form in $\mathcal Q^{\pm}_{N,D,r}$. 
By \eqref{eq:sl2r-action-on-qforms}, the coefficient of $xy$ in $W_u Q$ equals
\begin{equation}
    -2a\beta\delta+b\beta\gamma\frac{N}{u}+b\alpha\delta u-2c\alpha\gamma N \equiv r \left( \beta \gamma \frac{N}{u} + \alpha \delta u \right) \pmod{2N}.
\end{equation}
Let 
\begin{equation} \label{eq:psi-def}
    \psi(u) = \beta \gamma \frac{N}{u}+\alpha \delta u = 1 + 2\beta \gamma \frac{N}{u} = -1 + 2\alpha\delta u
\end{equation}
so that $W_u$ maps $\mathcal Q_{N,D,r}^{\pm}$ to $\mathcal Q_{N,D,\psi(u)r}^{\pm}$.
By \eqref{eq:psi-def} we have
\begin{equation} \label{eq:psi-cong}
    \psi(u) \equiv 1\pmod{2N/u} \quad \text{ and } \quad \psi(u) \equiv -1 \pmod{2u}.
\end{equation}
Since $(u,N/u)=1$, these congruences uniquely define $\psi(u)$ modulo $2N$.
Furthermore, they show that $\psi(u)^2 \equiv 1\pmod{2N}$.
Define
\begin{equation}
    \Psi_{2N} = \psi(U_N) \subseteq \{x\in (\Z/2N\Z)^\times : x^2=1\}.
\end{equation}
A calculation involving \eqref{eq:psi-cong} shows that $\psi:U_N\to \Psi_{2N}$ is injective and that
\begin{equation}
    \psi(u\ast r) \equiv \psi(u) \psi(r) \pmod{2N}.
\end{equation}
Thus $\psi : U_N \to \Psi_{2N}$ is a group isomorphism.

\subsection{Weakly holomorphic Poincar\'e series}

We specialize to the case when the modular function $F$ is a weakly holomorphic Poincar\'e series.
This construction is similar to the one in Section~4 of \cite{AA}.
We start with the real analytic Poincar\'e series
\begin{equation} \label{eq:Fm-poincare-def}
    F_{m}(\tau,s):=\sum_{\gamma \in \Gamma_\infty \setminus \Gamma_0(N)} \phi_s(m\Im \gamma \tau) e(-m\Re \gamma \tau),
\end{equation}
where $m \in \Z_+$ and
\[
    \phi_s(y):= 2\pi \sqrt{y} \, I_{s-1/2}(2\pi y).
\]
This series converges absolutely and uniformly on compact subsets of $\re(s)>1$.

Let $\ep$ be a character of $U_N$
such that $\ep(N)=-1$ and let 
\[
    f_{m}(\tau, s)= \sum_{r\mid\mid N} \ep(r) F_m(W_r\tau, s).
\]
Then for each $u \in U_N$, we have 
\begin{equation} \label{eq:fm-Wq-eig}
    f_m(W_u \tau,s) = \ep(u) f_m(\tau,s).
\end{equation}
To obtain Ramanujan's mock theta functions, we will eventually specialize $(N,\ep)$ to one of the pairs
\begin{equation} \label{eq:N-ep-pairs}
    (6,\ep_3), \quad (18, \ep_3), \quad (24,\ep_3), \quad (42, \ep_2\ep_3\ep_7), \quad (60, \ep_2\ep_3\ep_5),
\end{equation}
where $\ep_p$ is defined on prime powers $r^n \in U_N$ by 
$\ep_p(r^n) = -1$ if and only if $r=p$.
For the special pairs \eqref{eq:N-ep-pairs} and others, we will show that $f_m(\tau,s)$ can be analytically continued to $\re(s)>\frac 34$ and that $f_m(\tau,1)$ is a weakly holomorphic modular form.

Let $S_k^\ep(N)$ (resp.~$M_k^{!,\ep}(N)$) denote the vector space of cusp forms $f$ (resp.~weakly holomorphic modular forms) of weight $k$ on $\Gamma_0(N)$ satisfying $f(W_u\tau)=\ep(u)f(\tau)$ for all $u\in U_N$.
Then we have the following theorem. 
\begin{proposition} 
\label{prop:poincare_principal}
    Suppose that all cusps of $\Gamma_0(N)$ are of the form $\pm\frac{1}{\ell}$ where $\ell|N$. Let $\ep$ be a character of $U_N$ such that $S_2^\varepsilon(N)=\lrc{0}$. Then 
    \begin{equation} \label{eq:fm-limit-def}
        f_m(\tau) := \lim_{s\rightarrow 1} \sum_{r||N} \ep(r)F_m(W_r\tau, s)\in M_0^{!,\ep}(N).
    \end{equation}
    Furthermore, for each $u \in U_N$ we have $f_m(W_u\tau) - \ep(u)q^{-m} = O(1)$. 
    The function $f_m(\tau)$ is holomorphic at each cusp of $\Gamma_0(N)$ that is not of the form $W_u\infty$ for some $u\in U_N$.
\end{proposition}
\begin{remark*}
    The cusp condition is likely unnecessary, but it simplifies the proof significantly. 
    By Proposition 2.6 of \cite{iwaniec-classical}, when $N \in \{6,18,24,42,60\}$, all cusps of $\Gamma_0(N)$ are of the form $\pm \frac 1\ell$, with $\ell\mid N$. 
    Computation in Sage shows that $S_2^\ep(N)=\lrc{0}$ for each of the pairs \eqref{eq:N-ep-pairs}.
\end{remark*}

To prove Proposition~\ref{prop:poincare_principal}, we follow the argument given in Section~4 of \cite{AA}.
The first step is to compute the Fourier expansion of $F_m(\tau,s)$ at each cusp of $\Gamma_0(N)$.
By assumption, each cusp of $\Gamma_0(N)$ is of the form $\pm\frac 1\ell$ for $\ell \mid N$, so it suffices to compute $F_m(R_\ell^{\pm} \tau,s)$ where $R_\ell^{\pm} = \ptmatrix 10{\pm \ell}1$.
Furthermore, $R_N^{+}$ and $R_{N}^{-}$ represent the same cusp, so when $\ell=N$ we need only consider $R_N^+$.
We write the coefficients in terms of the classical Kloosterman sum
\[
    S(m,n,c) = \sum_{\substack{d\bmod c \\ (c,d)=1}} e\pfrac{m\bar d + nd}{c}
\]
and the $I,J,$ and $K$-Bessel functions.

\begin{proposition} \label{prop:poincare-fourier}
    Let $\ell,\ell',m\in \Z_+$ with $N=\ell\ell'$.
    Then for $\re(s)>1$ we have
    \begin{align*}
        F_{m}(R_\ell^{\pm}\tau, s) &= 2\delta_{\pm \ell=N}\pi \sqrt{my} I_{s-1/2}(2\pi my) e(-mx)\\
        &+a_s(0)y^{1-s} + 2\sqrt{my} \sum_{n\ne 0} a_s(n) K_{s-1/2}\lrp{2\pi \abs{\frac{n}{\ell'}}y} e\left(\frac{n}{\ell'} x\right),
    \end{align*}
    where 
    \begin{align*}
        a_s(0) & =\frac{2\pi^{s+1}m^s}{(\ell')^2(s-1/2)\Gamma(s)} \sum_{\substack{0<c\\ \ell\mid c}}c^{-2s} \sum_{k (\ell')} e\pfrac{\mp ck}{N} S\lrp{-m\ell', ck, c\ell' },
        \\
        a_s(n) &= \frac{2\pi}{(\ell')^2} \sum_{\substack{0<c\\ \ell|c}}c^{-1}\sum_{k (\ell')} e\pfrac{\mp ck}{N} S\lrp{-m\ell', n+ck, c \ell'} 
        \begin{cases} 
        I_{2s-1}\pfrac{4\pi \sqrt{mn/\ell'}}{c} & \text{ if } n>0, \\
        J_{2s-1}\pfrac{4\pi \sqrt{m|n|/\ell'}}{c} & \text{ if } n<0.
        \end{cases}
    \end{align*}
\end{proposition}
\begin{proof}
First we consider the case where $\ell \neq N$. Notice that 
\[R_\ell^{\pm}\ptmatrix{1}{N/e}{0}{1}(R_\ell^{\pm})^{-1}\in \Gamma_0(N),\]
which shows that $F_{m}(R_\ell^{\pm}\tau, s)$ is periodic with period $\ell'$.
Furthermore, $\Gamma_\infty \backslash \Gamma_0(N)R_\ell^\pm$ is represented by the set 
\[
    \lrc{\gamma =\ptmatrix{a}{b}{c}{d} : \det \gamma =1, (d,Nc)=1, c\equiv \pm\ell d\bmod N, 1\leq a \leq c-1}.
\]
Since $F_{m}(R_\ell^{\pm}\tau, s)$ has period $\ell'$, we may write 
\begin{align*}
    F_{m}(R_\ell^{\pm}\tau, s) &= \sum_{n\in \Z} A(n,y,s) e\left(\frac{n}{\ell'}x\right), \\
A(n,y,s) &= \frac{1}{\ell'} \int_0^{\ell'} \sum_{\gamma \in \Gamma_\infty \backslash \Gamma_0(N)R_\ell^\pm} \phi(m\Im \gamma \tau) e(-m\Re \gamma \tau) e\lrp{-\frac{n}{\ell'}x} dx.
\end{align*}
Swapping the sum and integral, and replacing $x\rightarrow x-d/c$, we find that
\[A(n,y,s)=\frac{1}{\ell'} \sum_{\gamma \in \Gamma_\infty \backslash \Gamma_0(N)R_\ell^\pm}e\pfrac{-am+\frac{n}{\ell'}d}{c}\int_{d/c}^{\ell'+d/c}  \phi\pfrac{my}{c^2|\tau|^2}  e\lrp{\frac{m}{c^2}\frac{x}{|\tau|^2}-\frac{n}{\ell'}x} dx,\]
where we have used that $\Re(\gamma(\tau-\tfrac{d}{c}))=\tfrac{a}{c}-\frac{x}{c^2|\tau|^2}.$
Let $G(c,x)$ denote the integrand above. Then
\begin{align}
    A(n,y,s) 
    &= \frac{1}{\ell'} \sum_{\substack{0<c\\\ell \mid c}} \sum_{\substack{d\equiv \pm c/\ell (\ell')\\(c\ell',d)=1 \\ 0\le d < c\ell'}} e\pfrac{-am\ell'+nd}{c\ell'} \int_{-\infty}^{\infty} G(c,x) dx
\end{align}
using the fact that $\ell|(c,N)$, so $(Nc,d)=1$ if and only if $(c\ell',d)=1.$ 
We write the condition $d\equiv \pm c/\ell(\ell')$ as an exponential sum and rearrange to obtain
\begin{align*}
    A(n,y,s) 
    &= \frac{1}{(\ell')^2} \sum_{\substack{0<c\\\ell\mid c}} \sum_{k (\ell')} e\pfrac{\mp ck}{N} S(-m\ell', n+ck, c\ell') \int_{-\infty}^{\infty} G(c,x) dx.
\end{align*}
The integral can be calculated exactly as in the proof of Proposition~5 of \cite{AA}.

Observe that the argument holds in the case where $\ell=N$, except that we must instead compute the Fourier expansion of the Poincar\'e series with the $c=0$ term removed, that is, $F(\tau,s)-\phi_s(my)e(-mx)$.
\end{proof}

\begin{proof}[Proof of Proposition~\ref{prop:poincare_principal}]

The classical Kloosterman sum satisfies $S(m,n,c) \ll c^{1/2+\ep}$.
This bound, together with a straightforward argument involving estimates for the $I$ and $J$-Bessel functions (see Chapter~10 of \cite{dlmf}) shows that the Fourier expansion of Proposition~\ref{prop:poincare-fourier} yields the analytic continuation of $F_m(\tau,s)$ to $\re(s)>\frac 34$. 
See Section~4 of \cite{AA} for details.

Since the Bessel functions specialize to $2\sqrt y \, K_{1/2}(y) = \sqrt{2\pi}e^{-y}$ and $\sqrt{\pi y} \, I_{1/2}(y) = \sqrt 2 \, \sinh y$, we have, in the notation of Proposition~\ref{prop:poincare-fourier},
\begin{multline}\label{eq:poincare-fourier}
    F_m(R_\ell^{\pm}\tau, 1) = \delta_{|\ell|=N} (e(-m\tau)-e(-m\bar\tau)) + a_1(0)\\ +\sqrt{m\ell'} \, \sum_{n>0} \frac{a_1(n)}{\sqrt n} e\lrp{\frac{n}{\ell'}\tau} + \sqrt{m\ell'} \, \sum_{n < 0} \frac{a_1(n)}{\sqrt{|n|}} e\lrp{\frac{n}{\ell'}\overline{\tau}}.
\end{multline}
Then $\xi_0 F_m(\tau, 1)$ is a cusp form of weight $2$, where $\xi_0 = \overline{\partial_{\bar z}}$. Thus $\xi_0 f_m(\tau) \in S_2^{\ep}(N)=\{0\}$.
It follows that $f_m(\tau)$ is weakly holomorphic.
Finally, \eqref{eq:poincare-fourier} shows that $f_m(\tau)$ has the desired behavior at each cusp.
\end{proof}

In the case $N=18$ we evaluate $a_1(0)$ for later use.

\begin{lemma} \label{lem:const-18}
    Suppose that $N=18$, that $m=1$, and that $\ell \in \{3,6\}$. Then, in the notation of Proposition~\ref{prop:poincare-fourier}, we have
    \[
        a_1(0) =
        \begin{cases}
            2 \zeta_3^{\mp 1} & \text{ if } \ell = 3, \\
            - \zeta_3^{\pm 1} & \text{ if } \ell = 6.
        \end{cases}
    \]
\end{lemma}

\begin{proof}
    By opening the Kloosterman sum, evaluating the $k$-sum, and replacing $c$ by $\ell c$, we find that
    \[
        a_1(0) = \frac{2\pi^2}{9\ell} \sum_{\substack{c>0}} c^{-2} \sum_{\substack{d(18c)^\times \\ d\equiv \pm c (\ell')}} e\pfrac{-\bar d}{\ell c} = \frac{2\pi^2}{9\ell} \sum_{\substack{c>0 \\ (\ell',c) = 1}} c^{-2} \sum_{\substack{d(18c) \\ d\equiv \pm \bar c (\ell')}} e\pfrac{-d}{\ell c} \sum_{k\mid (d,18c)} \mu(k),
    \]
    where $d\bar d\equiv 1\pmod{18c}$ and $c\bar c \equiv 1\pmod{\ell'}$.
    To satisfy $(d,\ell') = 1$, we need $(k,\ell')=1$, so the condition on the rightmost sum can be replaced by $k\mid (d,tc)$, where $t=\ell/3$.
    If we replace $d$ by $kd$ and write $d=\pm \overline{kc} + \ell'n$, we obtain
    \[
        a_1(0) = \frac{2\pi^2}{9\ell} \sum_{\substack{c>0 \\ (\ell',c) = 1}} c^{-2} \sum_{k\mid tc} \mu(k) e\pfrac{\mp \overline{kc}}{\ell c/k} \sum_{n(\ell c/k)} e\pfrac{-\ell'n}{\ell c/k}.
    \]
    The latter sum vanishes unless $\ell c/k$ divides $\ell'$, and this only holds when $k=tc$. In this case $\ell c/k = \ell/t = 3$. Since $c^2 \equiv 1\pmod{3}$ and $\bar t \equiv t \pmod 3$, it follows that
    \[
        a_1(0) = \frac{2\pi^2}{9 t} \zeta_3^{\mp t} \sum_{\substack{c>0 \\ (\ell',c) = 1}} \frac{\mu(tc)}{c^2}.
    \]
    By a standard argument involving the Euler product of $1/\zeta(2)$, the latter sum evaluates to $9\mu(t)/\pi^2$, and the lemma follows.
\end{proof}

\subsection{Traces of CM values of Poincar\'e series} \label{sec:traces-of-poincare}
Our aim is to apply Theorem~\ref{thm:Alfes-Trace} when $F=f_m$, then to relate the completions of the mock theta functions to $\mathcal I^M(f_m,\tau)$.
This was done in \cite{bruinier-schwagenscheidt} and \cite{kk} for each of Ramanujan's mock theta functions (with explicit eta quotients in place of $f_m$) except for the order 3 functions $\phi, \psi,$ and $\chi$.
For completeness, we give a sketch of the method here; see those papers for more details.

For simplicity, let $t_m(D,r)$ denote the quantity \eqref{eq:alfes-trace} when $F=f_m$, i.e.
\begin{equation}
    t_m(D,r) = i|D|^{-1/2}\left(\tr_{f_m}^+(D,r) - \tr_{f_m}^-(D,r)\right).
\end{equation}

\begin{theorem} \label{thm:vec-value-coeffs}
    Let $(N,\ep)$ be one of the pairs in \eqref{eq:N-ep-pairs} and let $D,r$ be as in Theorem~\ref{thm:Alfes-Trace}. If $N=18$, additionally assume that $m=1$.
    Suppose that $H = \sum_{r} H_r \mathfrak e_r \in \mathcal H_{1/2}(\rho_{L(N)})$ has principal part 
    \begin{equation} \label{eq:match-lift-principal}
        q^{-m^2/4N} \sum_{h\in \Psi(2N)}\varepsilon(\psi^{-1}(h))\mathfrak{e}_{mh} + 
        \begin{cases}
            -\mathfrak e_{12} + \mathfrak e_{-12} & \text{ if $N=18$ and $m=1$}, \\
            0 & \text{ otherwise}.
        \end{cases}
    \end{equation}
    Then the coefficient of $q^{|D|/4N}$ in the Fourier expansion of $H_r$ equals
    \begin{equation} \label{eq:mu-poincare-traces}
        \tfrac 12\sum_{v\mid m} \mu(m/v) t_v(D,r).
    \end{equation}
\end{theorem}

\begin{proof}
    The first step is to compute the principal part of the Millson theta lift $\mathcal I^M(f_m,\tau)$ using Theorem~5.1 of \cite{alfes-schwagenscheidt}.
    We begin by briefly recalling the notation of Sections~2 and 5 of that paper (more details can be found there).

    Let $V$ denote the rational quadratic space
    \[
        V=\lrc{X = \ptMatrix{x_2}{x_1}{x_3}{-x_2}: x_1,x_2, x_3\in \Q}
    \] 
    with the associated bilinear form $\langle X, Y\rangle = -N\tr(XY)$ and quadratic form $Q(X) = N\det(X).$ 
    The cusps of $\Gamma=\Gamma_0(N)$ are associated with the isotropic lines $\text{Iso}(V)$ in $V$ via the map\footnote{We suspect that there is a small sign error in the definition of $\psi$ in Section~2.1 of \cite{alfes-schwagenscheidt} since the relation $\psi(g(\alpha:\beta)) = g\psi((\alpha:\beta))$ does not hold with the definition given there. Here the minus sign is in the top left corner instead of the bottom right.} 
    \[
        \psi((\alpha:\beta)) = \text{span}\lrp{\ptMatrix{-\alpha\beta}{\alpha^2}{-\beta^2}{\alpha\beta}}.
    \]
    The cusp at $\infty$ is associated to $\lambda_\infty = \text{span}\lrp{\ptmatrix{0}{1}{0}{0}}.$ 
    For $\lambda\in \text{Iso}(V)$, let $\sigma_\lambda\in \SLtZ$ denote a matrix sending $\lambda_\infty$ to $\lambda$ via conjugation. 
    
    Let $\Lambda$ and $\Lambda'$ denote the lattice 
    \[
        \Lambda = \lrc{ \ptMatrix{b}{-a/N}{c}{-b}:a,b,c\in \Z},
    \] 
    and its dual lattice 
    \[
        \Lambda' = \lrc{ \ptMatrix{b/2N}{-a/N}{c}{-b/2N}:a,b,c\in \Z}.
    \]
    Then $\Lambda'/\Lambda \cong \Z/2N\Z$.
    For $h\in \Z/2N\Z$, let
    \[
        \hat h = \ptMatrix {h/2N}00{-h/2N}.
    \]
    It is straightforward to check that the map
    $h/2N \pmod \Z \mapsto \hat h \pmod \Lambda$
    from $L'(N)/L(N)$ to $\Lambda'/\Lambda$ induces an isomorphism $\mathcal H_k(\rho_{L(N)}) \to \mathcal H_k(\rho_\Lambda)$ that respects the bilinear forms on each side, in which $\mathfrak e_h$ is simply replaced by $\mathfrak e_{\hat h}$.
    
    For $n\in \Q_{<0}$ and $h\in \Z/2N\Z$, let
    \[
        \Lambda_{n,h} = \{ X\in \Lambda+\hat h: Q(x) = n\}.
    \]
    If $n=-Nd^2$ with $d\in \Q$ then each $X\in \Lambda_{n,h}$ is orthogonal to two lines $\lambda_X=\text{span}(Y)$ and $\lambda_{-X} = \text{span}(\tilde{Y}).$ We distinguish $\lambda_X$ from $\lambda_{-X}$ by choosing $\lambda_X$ such that
    \[\sigma_{\lambda_X}^{-1}X = d\ptMatrix{1}{-2r_{\lambda_X}}{0}{-1}\]
    for some $r_{\lambda_X}$.

    Let $F\in \mathcal H_0(N)$. Then, by Theorem $5.1$ of \cite{alfes-schwagenscheidt}, the principal part of $\mathcal{I}^M(F,\tau)$ is given by 
    \begin{equation} \label{eq:lift-pp}
        \sum_{h(2N)} \sum_{d>0} \frac{1}{2Nd} t^c(F; -Nd^2, h)q^{-Nd^2} \mathfrak e_h + \sum_{h(2N)} C_h \mathfrak e_h,
    \end{equation}
    for some $C_h\in \C$, where $t^c$ is the complementary trace
    \[
        t^c(F;-Nd^2, h) = \sum_{X\in \Gamma \backslash \Lambda_{-Nd^2,h}} \lrp{\sum_{w<0}a_{\lambda_X}^+(w) e(wr_{\lambda_X}) - a_{\lambda_{-X}}^+(w) e(wr_{\lambda_{-X}})},
    \] 
    and $a_{\lambda}^+(w)$ is the coefficient of $q^w$  of the holomorphic part of $F(\sigma_{\lambda_X}\tau)$.

    When $F=F_m$, by Proposition~\ref{prop:poincare-fourier} we find that $a_{\lambda}^+(w)=1$ if $\lambda=\lambda_\infty$ and $w=-m$, and $0$ otherwise. 
    Also, if $\lambda_{\pm X}=\lambda_\infty$ then $X$ is of the form $\ptmatrix{\pm d}{-a/N}{0}{\mp d},$ with $r_{\lambda_X}=\mp \tfrac{a}{2Nd}$. 
    Thus 
    \[
        t^c(F;-Nd^2, h) =  
        \begin{dcases}
            \pm\sum_{a(2Nd)} e\pfrac{\mp ma}{2Nd} & \text{ if } \frac h{2N} \equiv \pm d\pmod 1,\\
            0 & \text{ otherwise.}
        \end{dcases}
    \]
    If we write $d=\frac v{2N}$ with $v\in \Z$
    then the first term of \eqref{eq:lift-pp} equals
    \[
        \sum_{v|m} q^{-v^2/4N} (\mathfrak{e}_v -\mathfrak{e}_{-v}).
    \]

    By Theorem~5.1 of \cite{alfes-schwagenscheidt}, the constant $C_h$ in \eqref{eq:lift-pp} equals zero unless $\hat h\notin \Lambda$ and $\lambda \cap (\Lambda+\hat h) \neq \emptyset$ for some $\lambda \in \Gamma\backslash \operatorname{Iso}(V)$.
    By assumption, the cusps of $\Gamma_0(N)$ are of the form $(1:\pm \ell)$ for $\ell\mid N$, and a computation shows that for $\lambda = \psi((1:\pm \ell))$, the set $\lambda \cap (\Lambda+\hat h)$ is nonempty if and only if $\operatorname{lcm}(2\ell,2\ell') \mid h$, where $\ell\ell'=N$.
    In that case, 
    \[
        \hat h_\lambda := \ptMatrix{-\frac{h}{2N}}{\mp \frac{h}{2\ell}}{\pm \frac{h}{2\ell'}}{\frac{h}{2N}} \in \lambda \cap (\Lambda+\hat h)
    \]
    and thus $\beta_\lambda = 1/\ell$ and $k_\lambda = \mp h/2\ell N$.
    The formula for $C_h$ involves the factor $\cot(\pi k_\lambda / \beta_\lambda) = \cot(\mp \pi h/2N)$, which equals zero if $h=N$.
    Thus, the lcm condition above shows that for $N \in \{6, 18, 24, 42, 60\}$, the only $N$ for which some $C_h$ could be nonzero is $N=18$, with $C_{12}$ and $C_{24}$ possibly nonzero. In those cases, we have
    \[
        C_h = \frac{i}{6}  \sum_{\pm}  \cot\left(\mp \frac{\pi h}{36}\right) \sum_{\ell\in \{3,6\}} a_{\lambda}(0).
    \]
    By Lemma~\ref{lem:const-18} we have $C_{12} = -\frac 1{2}$ and $C_{24} = \frac 1{2}$.

    Write $\mathcal I^M(F,\tau) = \sum_h \mathcal I_{h}^M(F,\tau) \mathfrak e_h$.
    Then for each $r \in U_N$, by Proposition 3.2.5 of \cite{alfes} we have
    \[ 
        \mathcal{I}^M_{h'} (\tau, F \circ W_r) = \mathcal{I}^M_{h} (\tau, F),
    \]
    where $\hat h = W_r \hat h' W_r^{-1}$.
    In the notation of \eqref{eq:psi-def} we have
    \[
        h'=\lrp{\beta\gamma\frac{N}{r}+\alpha\delta r}h = \psi(r)h.
    \]
    From this and M\"obius inversion it follows that
    the principal part of $\mathcal{I}^M(\sum_{v|m} \mu(m/v)  f_v,\tau)$ equals
    \begin{equation}
        2q^{-m^2/4N} \sum_{h\in \Psi(2N)}\varepsilon(\psi^{-1}(h))\mathfrak{e}_{mh} + 
        \begin{cases}
            2(-\mathfrak e_{12} + \mathfrak e_{-12}) & \text{ if $N=18$ and $m=1$}, \\
            0 & \text{ otherwise}.
        \end{cases}
    \end{equation}

    Let $G = \frac 12\mathcal{I}^M(\sum_{v\mid m} \mu(m/v) f_v, \tau)-H$.
    Then the principal part of $G$ equals zero, so by Lemma~2.3 of \cite{bruinier-schwagenscheidt}, $G$ is a cusp form in $S_{1/2}(\rho_{L(N)})$. 
    But this space is isomorphic to the space $J_{1,N}^{\text{cusp}}$ of Jacobi cusp forms of weight $1$ and index $N$ (see Theorem~5.1 of \cite{EZ}), and that space is zero by Satz~6.1 of \cite{Skoruppa}.
    Thus $G=0$, and the theorem follows from Theorem~\ref{thm:Alfes-Trace}.
\end{proof}

\subsection{Mock theta coefficients as traces}
\label{sec:mock-theta-traces}

Theorem~\ref{thm:vec-value-coeffs} yields a formula for the mock theta function coefficients of the form
\begin{equation}\label{eq:formula-mobius-tvDr}
    \alpha_{\vartheta}(n) = \tfrac{1}{2}c_n\sum_{v\mid m} \mu(m/v) t_v(D,r)
\end{equation}
where $D$ is a linear function of $n$ and $c_n\in \C$.
We simply need to identify the data $N, \ep, m, D, r$ for each function, which we do case-by-case, using the lemmas in Section~\ref{sec:mock-theta-vectors}.
We work out the details in one case here;
the complete data is collected in Tables~\ref{tab:order-3}, \ref{tab:order-5-1}, and \ref{tab:order-7}.

Suppose that $N=60$.
Then $\Psi(2N) = \{1, 31, 41, 49, 71, 79, 89, 119\}$, and $\psi^{-1}$ maps these elements to the unitary divisors $\{1, 4, 3, 5, 12, 20, 15, 60\}$, respectively. 
By Lemma~\ref{lem-order5}, the principal part of ${H}_{(5,1)}$ equals
\begin{equation} \label{eq:H51-principal-part}
    q^{-1/60}\left(\sum_{r\in \{2, 22, 38, 58\}} \mathfrak e_r - \sum_{r\in \{62, 82, 98, 118\}} \mathfrak e_r \right),
\end{equation}
while, for $m=2$, the quantity in \eqref{eq:match-lift-principal} equals
\begin{equation} \label{eq:principal-part-to-match}
    q^{-1/60} \sum_{r\in \Psi(2N)} \ep(\psi^{-1}(r)) \mathfrak e_{2r}.
\end{equation}
A calculation using $\ep=\ep_2\ep_3\ep_5$ shows that $\ep(\psi^{-1}(r)) = 1$ if and only if $r\in \{1,71,79,89\}$, so \eqref{eq:H51-principal-part} equals \eqref{eq:principal-part-to-match}.
It follows from Theorem~\ref{thm:vec-value-coeffs} that
\[
    \alpha_{f_0}(n) = \tfrac 12\sum_{v\mid 2} \mu(2/v) t_v(4-240n,2).
\]

\begin{table}[h]
\centering
\caption{Order $3$ mock theta functions. All have $m=1$ and $\ep=\ep_3.$ 
}
\label{tab:order-3}
\renewcommand{\arraystretch}{1.2}
\begin{tabular}{|c|l|c|c|c|c|}
\hline
Function & $n$ & $N$  & $c_n$ & $-D$ & $r$  \\ \hline \hline
$f$ & all & $6$   & $1$ &  $24n-1$ & $1$ \\ \hline
$\omega$ & $\text{even}$ & $6$ & $ -\frac14$ &  $12n+8$ & $4$ \\ 
 & $\text{odd}$ &  & $-\frac14$ &   & $2$ \\ \hline
 $\chi$ & $0\pmod 3$  & $18$  & $1$ &  $24n-1$ & 1\\
 & $1\pmod 3$ &   & $1$ &   & 7 \\ 
 & $2\pmod 3$  &   & $-1$ &   & 5 \\ \hline
$\phi$ & $0\pmod 4$ & $24$   & $1$ &  $24n-1$ & $1$ \\ 
& $1\pmod 4$ &    & $-1$ &   & $13$ \\
& $2\pmod 4$ &  & $1$ &   & $7$ \\
& $3\pmod 4$ &  & $1$ &   & $5$ \\ \hline
$\psi$ & all & $24$   & $-1$ &  $96n-4$ & $2$ \\ \hline
\end{tabular}
\end{table}

\begin{table}[h]
\centering
\caption{Order $5$ mock theta functions. The left family has $m=1$ and the right family has $m=2$. All have $N=60$ and $\ep=\ep_2\ep_3\ep_5$.
}
\label{tab:order-5-1}
\renewcommand{\arraystretch}{1.2}
\begin{tabular}{|c|l|c|c|c|}
\hline
Function & $n$      & $c_n$     & $-D$   & $r$  \\ \hline \hline
$\psi_0$ & all      & $-1$       &  $240n-4$          & $2$ \\ \hline
$\psi_1$ & all      & $-1$       &  $240n+44$        & $14$ \\ \hline
$\phi_0$ & even  & $1$       &  $120n-1$ & $1$ \\ 
      & odd         & $-1$      &   & $11$ \\ \hline
$\phi_1$ & even  & $-1$      &  $120n-49$ & $7$ \\ 
      & odd         & $1$       &   & $13$ \\ \hline
\end{tabular}
\qquad
\begin{tabular}{|c|l|c|c|c|}
\hline
Function & $n$      & $c_n$     & $-D$   & $r$  \\ \hline \hline
$f_0$ & all         & $1$       &  $240n-4$          & $2$ \\ \hline
$f_1$ & all         & $1$       &  $240n+44$        & $14$ \\ \hline
$F_0$ & even        & $-\frac 12$ &  $120n-1$ & $1$ \\ 
      & odd         & $-\frac 12$ &   & $11$ \\ \hline
$F_1$ & even        & $-\frac 12$ &  $120n+71$ & $13$ \\ 
      & odd         & $-\frac 12$ &   & $7$ \\ \hline
\end{tabular}
\end{table}

\begin{table}[h]
\centering
\caption{Order $7$ mock theta functions. All have $m=1$, $N=42$, and $\ep=\ep_2\ep_3\ep_7$.
}
\label{tab:order-7}
\renewcommand{\arraystretch}{1.2}
\begin{tabular}{|c|l|c|c|c|}
\hline
Function & $n$          & $c_n$     & $-D$   & $r$  \\ \hline \hline
$\mathcal F_0$ & all    & $1$       &  $168n-1$          & $1$ \\ \hline
$\mathcal F_1$ & all    & $-1$       &  $168n-25$          & $5$ \\ \hline
$\mathcal F_2$ & all    & $-1$       &  $168n+47$          & $11$ \\ \hline
\end{tabular}
\end{table}

\section{From algebraic to transcendental formulas}
\label{sec:transcendental-formulas}

In this section we obtain transcendental formulas from the algebraic formulas of the previous section.
Let $\vartheta$ be one of the mock theta functions in Section~\ref{sec:mock-theta-vectors} and let $N,\ep,m,D,r$ be the corresponding data from Section~\ref{sec:mock-theta-traces}.
By \eqref{eq:psi-cong} we have $\psi(N) = -1$, so the action of $W_N$ on quadratic forms yields the bijection
\begin{equation}
    \mathcal Q_{N,D,r}^+ \xrightarrow{Q \mapsto -(W_N Q)} \mathcal Q_{N,D,r}^-.
\end{equation}
This, together with \eqref{eq:fm-Wq-eig}, shows that
\[
    t_m(D,r) = 2i|D|^{-1/2}\tr_{f_m}^+(D,r)
\]
since $\ep(N)=-1$.

Inserting the definitions \eqref{eq:tr-F-pm-def}, \eqref{eq:Fm-poincare-def}, \eqref{eq:fm-limit-def}, and applying \eqref{eq:zgq=gzq}, we find that
\begin{equation}
    t_m(D,r) 
    = 2i|D|^{-\frac12}\lim_{s\to 1^+} \sum_{u\mid\mid N} \sum_{Q\in \Gamma_0(N) \backslash \mathcal Q_{N,D,r}^+}  \sum_{\gamma \in \Gamma_\infty \backslash \Gamma_0(N)} \frac{\ep(u)}{\omega_Q} \phi_s(m\im \tau_{\gamma W_uQ}) e(-m\re \tau_{\gamma W_uQ}).
\end{equation}
As the pair $(\gamma,Q)$ ranges over $\Gamma_\infty \backslash \Gamma_0(N) \times \Gamma_0(N)\backslash \mathcal Q_{N,D,r}^+$, the product $\gamma W_u Q$, for a fixed $u$, ranges over the subset of $\Gamma_\infty \backslash \mathcal Q_{N,D}^+$ comprising those quadratic forms $Q=[a,b,c]$ with $b\equiv r\psi(u)\pmod{2N}$.
Thus
\begin{equation}
    t_m(D,r) = 2i|D|^{-\frac12}\lim_{s\to 1^+} \sum_{\substack{Q\in \Gamma_\infty \backslash \mathcal Q_{N,D}^+ \\ Q=[a,b,c]}} \sum_{\substack{u\mid\mid N \\ r\psi(u) \equiv b(2N)}} \ep(u) \phi_s(m\im \tau_{Q}) e(-m\re \tau_{Q}).
\end{equation}
The factor of $\omega_Q$ was eliminated because
\[
    \gamma W_u Q = \gamma' W_{u} Q' \iff Q = gQ' \text{ for some }g\in \Gamma_0(N).
\]
Next, because $\ptmatrix 1k01 [a,b,*] = [a,b-2ka,*]$, there is a bijection 
\[
    \Gamma_\infty \backslash \mathcal Q_{N,D}^+ \longleftrightarrow \{(a,b) \in \Z^2 : a>0, N\mid a, 0\leq b < 2a, \  b^2 \equiv D \!\!\! \pmod{4a}\}.
\]
It follows that
\begin{align}
    t_m(D,r) 
    &= 2i|D|^{-\frac12}\lim_{s\to 1^+} \sum_{N|a>0} \phi_s\pfrac{m|D|^{\frac12}}{2a}  \sum_{\substack{b(2a) \\ b^2 \equiv D (4a)}} \sum_{\substack{u\mid\mid N \\ r\psi(u) \equiv b(2N)}} \ep(u) e\pfrac{mb}{2a} \\
    \label{eq:ac-kloos-1}
    &= \frac{2\sqrt 2 \, \pi im^{\frac 12}}{|D|^{\frac14}N^{\frac 12}}\lim_{s\to 1^+} \sum_{c>0} c^{-\frac12} I_{s-\frac 12}\pfrac{\pi m|D|^{\frac12}}{Nc} \sum_{u\mid\mid N} \ep(u) \sum_{\substack{b(2Nc) \\ b\equiv r\psi(u)(2N) \\ b^2 \equiv D (4Nc)}} e\pfrac{mb}{2Nc}.
\end{align}

To emphasize that $D$ depends on $n$, we will sometimes write $D=-d_n$.
Note that this $d_n$ agrees with the quantity of the same name appearing in the introduction.
Define
\begin{equation} \label{eq:Acn-epsilon-def}
    A_c^{(m)}(n|\vartheta) = i\sqrt{\frac{c}{2N}}\sum_{u\mid\mid N} \ep(u) \sum_{\substack{b(2Nc) \\ b\equiv r\psi(u)(2N) \\ b^2 \equiv -d_n (4Nc)}} e\pfrac{mb}{2Nc}.
\end{equation}
We will show in Section~\ref{sec:kloo-simp} that for all $\vartheta$ except for those in the second fifth order family, the sum $A_c^{(1)}(n|\vartheta)$ agrees with the sum $A_c(n|\vartheta)$ appearing in Theorems~\ref{thm:third}, \ref{thm:fifth}, and \ref{thm:seventh}.

\begin{lemma} \label{lem:real}
    The sum $A_c(n|\vartheta)$ is real-valued.
\end{lemma}

\begin{proof}
    Replace $u$ by $N\ast u$ and observe that $\psi(N\ast u) \equiv \psi(N)\psi(u) \equiv -\psi(u) \pmod{2N}$. 
    Then the replacement $b\mapsto -b$ yields the relation $\overline{A_c(n|\vartheta)} = A_c(n|\vartheta)$ since $\ep(N)=-1$.
\end{proof}

\subsection{Kloosterman sums for the Weil representation}
The sums $A_c(n|\vartheta)$ are each linear combinations of Kloosterman sums for the Weil representation, which are defined as follows.
With the notation of Section~\ref{sec:weil-rep}, suppose that $k$ satisfies \eqref{eq:sigma-consistency}.
Then for $c \in \Z$, $c\neq 0$, $\frac{m}{2\Delta}\in \Z+q(\alpha)$, and $\frac{n}{2\Delta}\in \Z+q(\beta)$, 
we define
\begin{equation} \label{eq:kloo-weil-def-2}
    S_{\alpha,\beta}(m,n,c) = e^{-\pi i\sgn(c)k/2} \sum_{d(c)^\times} \bar \rho_{\alpha\beta}(\tilde\gamma) e\left(\frac{ma+nd}{2\Delta c}\right).
\end{equation}
Here $\gamma=\ptmatrix a*cd\in \SL_2(\Z)$ is any matrix with bottom row $(c\ d)$ and $\tilde\gamma = (\gamma,\sqrt{c\tau+d})$ is a lift of $\gamma$ to $\Mp_2(\Z)$.
Lemma~1.13 of \cite{bruinier} gives the relation
\begin{equation} \label{eq:kloo-neg-c}
    S_{\alpha,\beta}(m,n,c) = \overline{S_{\alpha,\beta}(m,n,-c)}.
\end{equation}
By Theorem~2.1 of \cite{AAW} we have
\begin{align*}
     A_c^{(m)}(n|\vartheta)
    &= - \sum_{q\mid\mid N} \ep(q) \sum_{u\mid (m,c)} \sqrt{u} \, S_{\frac{\psi(q)m}{2Nu}, -\frac{r}{2N}}\left(\frac{m^2}{u^2},D,\frac cu\right) \\
    &= - \sum_{u\mid(m,c)} \sqrt u \, \mathcal S_{r,\ep}\left(\frac mu,D,\frac cu\right),
\end{align*}
where
\begin{equation}
    \mathcal S_{r,\ep}(u,D,c) = \sum_{q\mid\mid N} \ep(q) S_{\frac{\psi(q)u}{2N}, -\frac{r}{2N}}(u^2,D,c).
\end{equation}
In this notation, we have
\begin{equation}\label{eq:trace-Kloosterman-1}
    t_m(D,r) 
    = -\frac{4\pi}{|D|^{\frac 14}} \lim_{s\to 1^+} \sum_{u\mid m} \sqrt{u} \sum_{c>0} \frac{\mathcal S_{r,\ep}(u,D,c)}{c} I_{s-\frac 12}\pfrac{\pi u |D|^{\frac12}}{Nc}.
\end{equation}

By M\"obius inversion, we find that
\[\mathcal{S}_{r,\ep}(m,D,c) = -\sum_{v|(m,c)} 
\mu(v)\sqrt{v} A^{(m/v)}_{c/v}(n|\vartheta),\]
and so by Lemma~\ref{lem:real}, $\mathcal S_{r,\ep}(u,D,c)$ is real-valued.
It follows from Corollary~\ref{cor:gs} below that
\begin{equation}
    \sum_{c\leq x} \frac{\mathcal{S}_{r,\ep}(u,D,c)}{c} \ll x^{1/2-\delta}
\end{equation}
for some $\delta>0$.
Therefore a straightforward argument (see Section~5 of \cite{AA} for details in a similar case) shows that the series in \eqref{eq:trace-Kloosterman-1} converges uniformly for $s\in [1,2]$, and this is enough to justify the exchange of the limit and sum to get
\begin{equation}
    t_m(D,r) = -\frac{4\pi}{|D|^{\frac 14}} \sum_{u\mid m} \sqrt{u} \sum_{c>0} \frac{\mathcal S_{r,\ep}(u,D,c)}{c} I_{\frac 12}\pfrac{\pi u |D|^{\frac12}}{Nc}.
\end{equation}
This, together with \eqref{eq:formula-mobius-tvDr}, yields the following theorem.

\begin{theorem} \label{thm:table-formulas}
    Let $\vartheta$ be one of the mock theta functions in Section~\ref{sec:mock-theta-vectors}, with corresponding data $N,\ep, m, r, c_n, D=-d_n$ from Section~\ref{sec:mock-theta-traces}.
    Then
    \begin{align}
        \alpha_\vartheta(n) 
        &= \frac{-2\pi c_n \sqrt{m} }{\sqrt[4]{d_n}}  \sum_{c>0} \frac{\mathcal S_{r,\ep}(m,-d_n,c)}{c} I_{\frac 12}\pfrac{\pi m \sqrt{d_n}}{Nc} \\
        &= \frac{2\pi c_n}{\sqrt[4]{d_n}} \sum_{v\mid m}\mu(m/v)\sqrt v \sum_{c>0} \frac{A_c^{(v)}(n|\vartheta)}{c} I_{\frac 12} \pfrac{\pi v \sqrt {d_n}}{Nc}.
    \end{align}
\end{theorem}

\subsection{Proof of Theorems~\ref{thm:third}, \ref{thm:fifth}, and \ref{thm:seventh}} \label{sec:kloo-simp}

To prove the three theorems in the introduction, it remains to simplify the sums $A_c^{(1)}(n|\vartheta)$ (note that $m=1$ for each of the mock theta functions in those theorems).
We observe that $b$ only contributes to the sum \eqref{eq:Acn-epsilon-def} if $b\in r\Psi(2N)$.
Thus we can write
\[
    A_c^{(1)}(n|\vartheta) = i \sqrt{\frac{c}{2N}} \sum_{\substack{b(2Nc) \\ b\in r\Psi(2N) \\ b^2\equiv -d_n (4Nc)}} \xi(b) e\pfrac{b}{2Nc},
\]
where
\begin{equation}
    \xi(b) = \sum_{\substack{q\mid\mid N \\ r\psi(q) \equiv b (2N)}} \ep(q).
\end{equation}
With the exception of the third order function $\psi(q)$, in each case the condition $b\in r\Psi(2N)$ is redundant because it is implied by the congruence $b^2 \equiv -d_n\pmod{4Nc}$.
For $\psi(q)$ the condition $b\in r\Psi(2N)$ is equivalent to $b\equiv 2,14,34,46\pmod{48}$.

For each of the pairs $(N,\ep)$ in \eqref{eq:N-ep-pairs}, we can compute $\xi(b)$ for each $b\in r\Psi(2N)$ for each $r$ appearing in the corresponding table from Section~\ref{sec:mock-theta-traces} (for now, we are only considering the left family from Table~\ref{tab:order-5-1}).
The results are somewhat cleaner if we compute the product $c_n \xi(b)$ instead. 
For the third order functions we find that
\begingroup
\allowdisplaybreaks
\begin{align}
    \text{if $N=6$ then }& c_n \xi(b) = \pfrac{-3}{b} \times
    \begin{cases}
        1 & \text{ if }r=1, \\
        -\tfrac 12 (-1)^n & \text{ if }r=2,4,
    \end{cases} \\
    \text{if $N=18$ then }& c_n\xi(b) = \pfrac{-3}{b} \text{ if } r=1,\\
    \text{if $N=24$ then }& c_n\xi(b) = \pfrac{-3}{b} \times
    \begin{cases}
        (-1)^n & \text{ if }r=1,5,7,13, \\
        1 & \text{ if }r=2. \\
    \end{cases}
\end{align}
\endgroup
For the fifth order functions, we have $N=60$ and
\begin{equation}
     c_n\xi(b) = \pfrac{3}{b}\chi_{5}(2,b) \times
    \begin{cases} 
        (-1)^n & \text{ if }r=1,11, \\
        -i & \text{ if }r=2, \\
        -i(-1)^n & \text{ if }r=7,13, \\
        1 & \text{ if }r=14. \\
    \end{cases}
\end{equation}
Finally, for the seventh order functions we have $N=42$ and
\begin{align}
     c_n\xi(b) = \pfrac{-21}{b} \times
     \begin{cases}
         1 & \text{ if }r=1, \\
         -1 & \text{ if }r=5,11.
     \end{cases}
\end{align}
The quantities in braces above exactly agree with the factors $k_n$ in Theorems~\ref{thm:third}, \ref{thm:fifth}, and \ref{thm:seventh}.
We complete the proof by combining the contributions from $b$ and $-b$ in the sum to replace the exponential factor with the sine function. \qed

\subsection{The second fifth order family}
The mock theta functions $f_0$, $f_1$, $F_0$, and $F_1$ have $m=2$, so the corresponding transcendental formulas are slightly more complicated than those given in the introduction.
The following theorem can be proved using the same ideas as in the previous subsection.

\begin{theorem} \label{thm:fifth-2}
    Let $\vartheta \in \{f_0, f_1, F_0, F_1\}$ and let $d_n$ and $k_n$ be given by the following table.
    \[
    \renewcommand{\arraystretch}{1.2}
        \begin{tabular}{|c|c|c|c|c|}
            \hline
             & $f_0$ & $f_1$ & $F_0$ & $F_1$ \\
             \hline
            $d_n$ & $240n-4$ & $240n+44$ & $120n-1$ & $120n+71$  \\
            \hline
            $k_n$ & ${i}$ & ${-1}$ & ${-\tfrac 12}$ & ${-\tfrac 12i}$ \\ \hline
        \end{tabular}
    \]
    Then
    \[
        \alpha_\vartheta(n) = \frac{2\pi}{\sqrt[4]{d_n}} \sum_{v\mid 2} \mu(2/v) \sqrt v \sum_{c>0} \frac{A_c^{(v)}(n|\vartheta)}{c} I_{1/2} \pfrac{\pi v\sqrt{d_n}}{60c},
    \]
    where
    \[
        A_c^{(m)}(n|\vartheta) = k_n \sqrt{\frac{c}{30}} \sum_{\substack{b\bmod 60c \\ b^2\equiv -d_n(240c)}} \pfrac{3}{b} \chi_{5}(2,b) \sin\pfrac{-mb}{60c}.
    \]
\end{theorem}

\section{The Goldfeld-Sarnak approach to sums of Kloosterman sums}
\label{sec:goldfeld-sarnak}

In this section we give an asymptotic formula for sums of the Kloosterman sums $S_{\alpha,\beta}(m,n,c)$ that generalizes the main result of \cite{GS} (see also \cite{pribitkin}).
The results given here are essentially contained in Appendix~E of \cite{hejhal}, especially pages 700--709.
Our setup matches Version~C of \cite{hejhal} (see the table on page~8) which is discussed in Chapter~9 of that book.
In an effort to keep the present work relatively self-contained, we sketch a proof of the sums of Kloosterman sums result here.

One of the central objects in this section is the weight $k$ hyperbolic Laplacian
\begin{equation} \label{eq:Deltak-def-2}
    \Delta_k = y^2 \left(\frac{\partial^2}{\partial x^2} + \frac{\partial^2}{\partial y^2}\right) - i k y \frac{\partial}{\partial x}.
\end{equation}
Note that this differs from the operator of the same name appearing in Section~\ref{sec:weil-rep}.
The normalization \eqref{eq:Deltak-def-2} is customarily used in the spectral theory of automorphic forms, while the other is used in the theory of harmonic Maass forms.
The operator \eqref{eq:Deltak-def-2} is the only one of the two used in this section, so this should not cause any confusion.

In the notation of Section~\ref{sec:weil-rep}, let $\alpha,\beta\in L'/L$ and let $m,n\in \Z$ with $\frac{m}{2\Delta}\in \Z+q(\alpha)$, $\frac{n}{2\Delta}\in \Z+q(\beta)$.
Since \eqref{eq:sigma-consistency} defines $k$ modulo $2$, we will assume that $k\in \{0, \pm \frac 12, 1\}$.
Define 
\[
    B = \limsup_{c\rightarrow \infty} \frac{\log|S_{\alpha,\beta}(m,n,c)|}{\log c}.
\] 
The following asymptotic formula is the analogue of Theorem~2 of \cite{GS}. 

\begin{theorem} \label{thm:gs}
    There exist exponents $e_1 > e_2 > \ldots > e_\ell\in (0,B]$ (depending only on $L$) and constants $a_1,\ldots,a_\ell\in \C$ (depending on $L, k, \alpha, \beta, m, n$) such that
    \begin{equation}
        \sum_{0 < |c|\leq x} \frac{S_{\alpha,\beta}(m,n,c)}{|c|} = \sum_{j=1}^\ell a_j x^{e_j} + O(x^{B/3+\varepsilon})
    \end{equation}
    for any $\varepsilon>0$.
\end{theorem}
A formula for the constants $a_j$ is given in \eqref{eq:aj-formula} below.

In the special case that $g=\operatorname{rank}(L)$ is odd and $m=m_0v^2$ with $(v,\Delta)=1$ and $(-1)^{(g-1)/2}m_0$ a fundamental discriminant, Theorem 1.1 in \cite{AAW} gives the Weil bound 
\begin{equation} \label{eq:weil-bound}
    S_{\alpha,\beta}(m,n,c) \ll c^{1/2+\ep}.
\end{equation}
This shows in particular that $B\le 1/2$ in these cases.
Furthermore, since $g$ is odd we have $k=\pm \frac{1}{2}$ (see \eqref{eq:sigma-consistency}).
We will show below that $e_j=2s_j-1$, where $s_j(1-s_j) = \lambda_j$ and $\lambda_j < \frac 14$ is the eigenvalue of a Maass cusp form.
For $k=\pm \frac 12$ we have $\lambda_1 \geq \frac 3{16}$ (see \eqref{eq:exceptional-ineq}), so $s_1 \leq \frac 34$ and $e_1 \leq \frac 12$.
The bottom eigenvalue $\lambda_1=\frac{3}{16}$ is present if and only if there is a nonzero holomorphic cusp form in $S_{\pm k}(\rho_L)$.
The formula \eqref{eq:aj-formula} shows that $a_j$ is proportional to the product of the $m$-th and $n$-th Fourier coefficients of the corresponding Maass cusp form.
Therefore, if $m$ and $n$ have opposite sign, we have $a_1=0$ if $e_1 = \frac 12$.
The following corollary (which also uses \eqref{eq:kloo-neg-c}) summarizes these results.

\begin{corollary} \label{cor:gs}
    Suppose that $g=\operatorname{rank}(L)$ is odd and that $m=m_0v^2$ with $(v,\Delta)=1$ and $(-1)^{(g-1)/2}m_0$ a fundamental discriminant.
    Suppose that $mn<0$. Then there exists a $\delta>0$ (depending only on $L$) such that
    \begin{equation}
        \sum_{c\leq x} \frac{\re S_{\alpha,\beta}(m,n,c)}{c} \ll x^{1/2-\delta}.
    \end{equation}
    The implied constant may depend on $\alpha,\beta,m,n$.
\end{corollary}

The proof of Theorem~\ref{thm:gs} relies on an estimate for the Kloosterman-Selberg zeta function 
\begin{equation}
    Z_{m,n}(s) = Z_{m,n}^{(\alpha,\beta)}(s) = \tfrac 12\sum_{c\neq 0} \frac{S_{\alpha,\beta}(m,n,c)}{|c|^{2s}}
\end{equation}
in the critical strip $\frac 12 < \sigma < 1$ as $t\to\infty$, where $s=\sigma+it$.
This bound is obtained by applying the spectral theory of $\Delta_k$
to the Poincar\'e series
\begin{equation} \label{eq:poincare-def}
    P_{\alpha,m}(\tau,s,k) = \tfrac 12 \sum_{\substack{\mathfrak g\in \widetilde{\Gamma}_\infty \backslash \tilde\Gamma \\ \mathfrak g = (\gamma,\phi)}} j(\mathfrak g, \tau)^{-2k} \rho(\mathfrak g)^{-1} e(\tilde m \gamma \tau) \im(\gamma\tau)^s\mathfrak{e}_\alpha,
\end{equation}
where $\tilde \Gamma = \Mp_2(\Z)$, $\tilde \Gamma_\infty = \langle T,Z^2 \rangle$, $j(\mathfrak g,\tau) = \phi(\tau)/|\phi(\tau)|$, and $\tilde m = \frac m{2\Delta}$.
This is a vector-valued analogue of the series studied by Selberg in \cite{selberg}; it satisfies the following properties.

\begin{proposition}
    The series \eqref{eq:poincare-def} converges absolutely and uniformly in $\re(s)>1$ and satisfies the recurrence relation
    \begin{equation} \label{eq:poincare-recurrence}
        (\Delta_k + s(1-s)) P_{\alpha,m}(\tau,s,k) = -4\pi \tilde m (s-\tfrac k2) P_{\alpha,m}(\tau,s+1,k).
    \end{equation}
    Furthermore, the function $P_{\alpha,m}(\tau,s,k)$ has a meromorphic continuation to $\re(s)>\frac 12$ with at most simple poles in the segment $(\frac 12, 1)$ on the real line.
\end{proposition}

\begin{proof}
    These facts can all be found in Chapter~9 of \cite{hejhal}.
    The series \eqref{eq:poincare-def} is (7.3) on page 415, see also (7.6) on page 416 and pages 355--356.  
    The recurrence relation \eqref{eq:poincare-recurrence} is a straightforward calculation.
    For the meromorphic continuation and simple poles, see the discussion at the bottom of page 414 together with (6.11) on page 321.
\end{proof}

Let $\mathcal L_k(\rho_L)$ denote the Hilbert space of functions $f = \sum_{\alpha} f_\alpha \mathfrak e_\alpha:\H\to \C[L'/L]$ satisfying $\langle f,f \rangle < \infty$ and $f(\gamma \tau) = j(\mathfrak g,\tau)^{2k} \rho_L(\mathfrak g) f(\tau)$ for all $\mathfrak g = (\gamma,\phi) \in \Mp_2(\Z)$.
Here $\langle \cdot, \cdot \rangle$ denotes the Petersson inner product
\begin{equation}
    \langle f, g \rangle = \int_{\mathcal D} (f(\tau), g(\tau)) \frac{dxdy}{y^2}, \qquad (f,g) = \sum_{\alpha\in L'/L} f_\alpha \bar g_\alpha,
\end{equation}
where $\mathcal D$ is a fundamental domain for $\SL_2(\Z) \backslash \H$.
The Poincar\'e series \eqref{eq:poincare-def} is in $\mathcal L_k(\rho_L)$ when $m>0$ (see Proposition~3 on page 668 of \cite{hejhal}).
The spectrum of $\Delta_k$ on $\mathcal L_k(\rho_L)$ is described in detail in Chapter~9 of \cite{hejhal}; see page 414, together with page 291 and pages 316--318 (especially (5.7)).
We collect the relevant facts here.

The spectrum of $\Delta_k$ is real and nonnegative.
In $[0,\frac 14)$ there are finitely many eigenvalues
\begin{equation} \label{eq:exceptional-ineq}
    \tfrac{|k|}{2}\left(1 - \tfrac{|k|}{2}\right) \leq \lambda_1 \leq \cdots \leq \lambda_p,
\end{equation}
corresponding to eigenfunctions $u_j \in \mathcal L_k(\rho_L)$, where
\[
    (\Delta_k + \lambda_j) u_j = 0.
\]
We may assume that the eigenfunctions $u_j$ form an orthonormal set, and for each $\lambda$, let
\[
    \mathcal C(\lambda) = \{u_j : (\Delta_k + \lambda) u_j = 0\}.
\]
Each $u_j$ has a Fourier expansion of the form
\begin{equation} \label{eq:uj-fourier-exp}
    u_j(\tau) = \sum_{\substack{\alpha\in L'/L \\ q(\alpha) \in \Z}} c_{\alpha j}(0) y^{1 - s_j}\e_\alpha + \sum_{\alpha\in L'/L} \sum_{\substack{n\in \Z+q(\alpha) \\ \tilde n\neq 0}} c_{\alpha j}(\tilde n) W_{\frac k2\sgn(\tilde n),s_j-\frac 12} (4\pi|\tilde n|y) e(\tilde nx) \e_\alpha,
\end{equation}
where $\lambda_j = s_j(1-s_j)$ and $\re(s_j)>\frac 12$.
The bottom eigenvalue $\lambda(k) = \frac{|k|}{2}(1-\frac{|k|}{2})$ appears if and only if there exists a holomorphic cusp form $F$ (or the constant function $1$ if $k=0$) of weight $|k|$ for $\rho_L$.
In that case, we have
\begin{equation} \label{eq:bottom-spectrum}
    u_1(\tau) = \frac{\tilde u_1(\tau)}{\lVert \tilde u_1(\tau) \rVert}, \qquad \tilde u_1(\tau) =  
    \begin{cases}
        y^{k/2} F(\tau) & \text{ if }k\geq 0, \\
        y^{-k/2} \bar F(\tau) & \text { if }k < 0.
    \end{cases}
\end{equation}

\begin{lemma} \label{lem:residue-poincare}
    If $m>0$ then
    \begin{equation}
        \res_{s=s_j} P_{\alpha,m}(\tau,s,k) = (4\pi)^{1-s_j} \frac{\Gamma(2s_j-1)}{\Gamma(s_j - \frac k2)} \sum_{u_\ell \in \mathcal C(\lambda_j)} \tilde m^{1-s_j} \bar c_{\alpha \ell}(\tilde m) u_\ell.
    \end{equation}
\end{lemma}

\begin{proof}
Let $g(\tau) = \res_{s=s_j} P_{\alpha,m}(\tau,s,k)$. Then $g\in \mathcal L_k(\rho_L)$ and \eqref{eq:poincare-recurrence} gives
    \begin{align*}
        (\Delta_k + \lambda_j)g(\tau) &= \lim_{s\to s_j} (s-s_j)(\Delta_k + s_j(1-s_j)) P_{\alpha,m}(\tau,s,k) \\
        & = -4\pi \tilde m \lim_{s\to s_j} (s-s_j)(s-k/2) P_{\alpha,m}(\tau,s+1,k) = 0
    \end{align*}
because $P_{\alpha,m}(\tau,s+1,k)$ is analytic at $s=s_j$.
Thus $g$ is an eigenfunction of $\Delta_k$ with eigenvalue $\lambda_j$.
It follows that
    \[
        g = \sum_{u_\ell\in \mathcal C(\lambda_j)} \langle g, u_\ell \rangle u_\ell = \res_{s=s_j} \sum_{u_\ell\in \mathcal C(\lambda_j)} \langle P_{\alpha,m}(\cdot,s,k), u_\ell \rangle u_\ell.
    \]
In order to compute the inner product in the expression above, we apply a standard unfolding argument: 
if $\re(s)$ is sufficiently large then
\begin{align}
    \langle P_{\alpha,m}(\cdot,s,k), u_\ell \rangle 
    &= \tfrac 12\sum_{\substack{\mathfrak g \in \tilde \Gamma_\infty \backslash \tilde \Gamma \\ \mathfrak g = (\gamma,\phi)}} \int_{\mathcal D} j(\mathfrak g,\tau)^{-2k} e(\tilde m \gamma\tau) \im(\gamma\tau)^s (\rho_L(\mathfrak g)^{-1} \e_\alpha,u_\ell(\tau)) \, \frac{dx dy}{y^2} \\
    &= \tfrac 12\sum_{\mathfrak g \in \tilde \Gamma_\infty \backslash \tilde \Gamma} \int_{\gamma\mathcal D} e(\tilde m \tau) y^s (\e_\alpha, j(\mathfrak g,\gamma^{-1}\tau)^{2k} \rho_L(\mathfrak g) u_\ell(\gamma^{-1}\tau)) \, \frac{dx dy}{y^2},
\end{align}
where we have made the change of variable $\tau\mapsto \gamma^{-1}\tau$ and we have used that $\rho_L$ is unitary.
Since $j(\mathfrak g,\gamma^{-1}\tau)^{2k} \rho_L(\mathfrak g)u_\ell(\gamma^{-1}\tau) = u_\ell(\tau)$, we have
\begin{align} 
   \langle P_{\alpha,m}(\cdot,s,k), u_\ell \rangle
    &= \tfrac 12\sum_{\mathfrak g \in \tilde \Gamma_\infty \backslash \tilde 
 \Gamma} \int_{\gamma\mathcal D} e(\tilde m \tau) y^s (\e_\alpha, u_\ell(\tau)) \, \frac{dx dy}{y^2} \\
    &= \int_{0}^\infty y^{s-2} e^{-2\pi \tilde m y}  \int_0^1 (e(\tilde m x)\e_\alpha,u_\ell(\tau))\, dx \,dy. \label{eq:poincare-inner-prod-integral}
\end{align}
The inner integral equals the complex conjugate of the $(\tilde m, \alpha)$ Fourier coefficient of $u_\ell$, including the Whittaker factor.
By \cite[(7.621.11)]{GR} we obtain
\begin{align} 
    \langle P_{\alpha,m}(\cdot,s,k), u_{\ell} \rangle
    &= \overline{c}_{\alpha \ell}(\tilde m) \int_0^\infty y^{s-2} e^{-2\pi \tilde m y} W_{\frac k2, s_j - \frac 12}(4\pi \tilde my) \, dy \\
    &= (4\pi\tilde m)^{1-s} \overline{c}_{\alpha \ell}(\tilde m) \frac{\Gamma(s-s_j)\Gamma(s+s_j-1)}{\Gamma(s-\frac k2)}. \label{eq:poincare-inner-prod-gammas}
\end{align}
It follows that
\begin{equation}
    g = (4\pi)^{1-s_j} \frac{\Gamma(2s_j-1)}{\Gamma(s_j-\frac k2)} \sum_{u_\ell \in \mathcal C(\lambda_j)} {\tilde m}^{1-s_j}  \overline{c}_{\alpha \ell}(\tilde m) u_\ell,
\end{equation}
and this completes the proof.
\end{proof}

We have the following representation of the Kloosterman-Selberg zeta function in terms of the inner product of two Poincar\'e series. Let $P^*_{\beta,n}(\tau,s,k)$ denote the sum \eqref{eq:poincare-def} with $\rho_L$ replaced by the dual representation $\rho_L^*$ (which is simply the complex conjugate if we think of $\rho_L(\mathfrak g)$ as a matrix with entries in $\C$).

\begin{lemma} \label{lem:poincare-inner-prod-opp-sign}
    Suppose that $m>0$ and $n<0$, with $\tilde m = \frac{m}{2\Delta}\in \Z+q(\alpha)$ and $\tilde n = \frac{n}{2\Delta} \in \Z+q(\beta)$.
    Then for $\sigma=\re s>1$ we have 
    \begin{equation*}
    \langle P_{\alpha,m}(\cdot,s,k), \overline{P^*_{\beta,-n}(\cdot,s+2,-k)}\rangle   
    =  \frac{\Gamma(2s+1)}{4^{s+1}\pi|\tilde n|^2 \Gamma(s-\frac k2)\Gamma(s+2+\frac k2)} Z_{m,n}(s) +R_{m,n}(s),
    \end{equation*}
where $R_{m,n}(s)$ is holomorphic and $O_{m,n}(1)$ in the larger region $\sigma > \frac 12$.
In particular, $Z_{m,n}(s)$ has a meromorphic continuation to $\sigma > \frac 12$.
\end{lemma}

\begin{proof}
First, suppose $\sigma>1$. 
By a similar unfolding arugument as in the proof of Lemma~\ref{lem:residue-poincare} we have
\begin{equation}
    \langle P_{\alpha,m}(\cdot,s,k), \overline{P^*_{\beta,-n}(\cdot,s+2,-k)}\rangle
    = \int_0^\infty y^{s} e^{2\pi \tilde n y} \int_0^1 \left( P_{\alpha,m}(\tau,s,k), e(\tilde n x) \e_\beta \right) \, dx dy.
\end{equation}
The inner integral can be evaluated by following the proof of Theorem~1.9 of \cite{bruinier} (see the bottom of page 31 to the top of page 33), and we find that
\begin{multline}
    y^s\int_0^1 \left( P_{\alpha,m}(\tau,s,k), e(\tilde n x) \e_\beta \right) \, dx \\
    = 
    \tfrac 12 i^k y\sum_{c \neq 0}\frac{S_{\alpha,\beta}(m,n,c)}{|c|^{2s}}\int_{-\infty}^\infty (x+i)^{-k}(x^2+1)^{k/2-s} e \left(\frac{-\tilde m}{c^2y(x+i)} - \tilde n xy\right) \, dx.
\end{multline}
The clever idea of Goldfeld and Sarnak is to write
\[
    e \left(\frac{-\tilde m}{c^2y(x+i)} \right) = 1 + \left(e \left(\frac{-\tilde m}{c^2y(x+i)} \right) - 1\right),
\]
so that
\begin{multline} \label{eq:poincare-Z-R-unevaluated}
    \langle P_{\alpha,m}(\tau,s,k), \overline{P^*_{\beta,-n}(\tau,s+2,-k)}\rangle 
    \\
    =  Z_{m,n}(s) \int_0^\infty y e^{2\pi \tilde n y} I_{n}(y,s) dy + \sum_{|c|\neq 0} \frac{S_{\alpha,\beta}(m,n,c)}{|c|^{2s}} R_{m,n}(s,c),
\end{multline}
where (by \cite[(3.384.9)]{GR})
\begin{align}
    I_n(y,s) &= i^k\int_{-\infty}^\infty (x+i)^{-k}(x^2+1)^{k/2-s} e(-\tilde n x y) \, dx \\
    &= \pi \frac{(\pi |\tilde n|y)^{s-1}}{\Gamma(s-\frac k2)} W_{-\frac k2, \frac 12 - s}(4\pi |\tilde n|y)
\end{align}
and
\begin{equation}
    R_{m,n}(s,c) = \tfrac 12 i^k \int_0^\infty y e^{2\pi \tilde n y} \int_{-\infty}^\infty\frac{(x+i)^{-k}}{(x^2+1)^{s-k/2}} \left(e\left(\frac{-\tilde m}{c^2y(x+i)}\right)-1\right)e \left(-\tilde nxy\right) dxdy.
\end{equation}
By \cite[(7.621.11)]{GR}, the first term on the right-hand side of \eqref{eq:poincare-Z-R-unevaluated} equals
\begin{equation}
    \frac{\Gamma(2s+1)}{4^{s+1}\pi|\tilde n|^2 \Gamma(s-\frac k2)\Gamma(s+2+\frac k2)} Z_{m,n}(s).
\end{equation}

We claim that the sum in \eqref{eq:poincare-Z-R-unevaluated} involving $R_{m,n}(s,c)$ converges absolutely and uniformly on compact subsets of $\re(s)>\frac 12$.
The quantity $\frac{-2\pi i\tilde m}{c^2 y(x+i)}$ has negative real part, and
for any $\tau$ with negative real part, we have $|e^\tau - 1| \ll |\tau|$.
Thus
\begin{equation}
    |R_{m,n}(s,c)| \ll |c|^{-2} \int_0^\infty e^{2\pi \tilde n y} \int_{-\infty}^\infty (x^2+1)^{-\sigma-1/2} \, dxdy  \ll |c|^{-2}
\end{equation}
for $\sigma>\frac 12$.
This shows that
\[R_{m,n}(s) = \sum_{|c| \neq 0}\frac{S_{\alpha,\beta}(m,n,c)}{|c|^{2s}} R_{m,n}(s,c)\]
is holomorphic and $O_{m,n}(1)$ in $\sigma>\frac 12$. 
\end{proof}

In the case where $m$ and $n$ have the same sign, we have the following similar formula expressing $Z_{m,n}(s)$ as an inner product of Poincar\'e series (see Lemma~2 of \cite{GS}). 
The proof is virtually the same as the proof of Lemma~\ref{lem:poincare-inner-prod-opp-sign}.

\begin{lemma} \label{lem:poincare-inner-prod-same-sign}
    For $m,n>0$ and $\sigma>1$ we have 
    \begin{multline*}
    \langle P_{\alpha,m}(\tau,s,k), P_{\beta,n}(\tau,\overline{s}+2,k)\rangle   
    = \delta_{m,n}\delta_{\alpha,\beta}\left(2\pi \tilde n\right)^{-2s-1} \Gamma(2s+1)\\ + \frac{\Gamma(2s+1)}{4^{s+1}\pi \tilde n^2\Gamma(s+k/2)\Gamma(s-k/2+2)} Z_{m,n}(s) +R_{m,n}(s),  
    \end{multline*}
where $R_{m,n}(s)$ is holomorphic and $O_{m,n}(1)$ in the larger region $\sigma>\frac 12$. In particular, $Z_{m,n}(s)$ has a meromorphic continuation to $\sigma > \frac 12$.
\end{lemma}

Using Lemmas~\ref{lem:residue-poincare}, \ref{lem:poincare-inner-prod-opp-sign}, and \ref{lem:poincare-inner-prod-same-sign} we compute
\begin{equation} \label{eq:residue-tau}
    \res_{s=s_j} Z_{m,n}(s) = \pi^{1-2s_j} (4\tilde m|\tilde n|)^{1-s_j} \frac{\Gamma(s_j+\sgn(\tilde n) \frac k2)\Gamma(2s_j-1)}{\Gamma(s_j - \frac k2)} \sum_{u_\ell\in \mathcal C(\lambda_j)} \bar c_{\alpha \ell}(\tilde m) c_{\beta \ell}(\tilde n).
\end{equation}

The key input into the proof of Theorem~\ref{thm:gs} is the following estimate for $Z_{m,n}(s)$ in the critical strip $\frac 12 < \sigma \leq 1$.

\begin{proposition}
    For $\frac 12 < \sigma \leq 1$ and $|t|\geq 1$ we have
    \begin{equation} \label{eq:Z-bound}
        Z_{m,n}(s)\ll \frac{|t|^{1/2}}{\sigma-\frac 12}.
    \end{equation}
\end{proposition}

\begin{proof}
    The Poincar\'e series with second argument equal to $s+2$ converges absolutely in the region $\frac 12 \leq \sigma \leq 1$ and is bounded in that region by a function that depends only on $\sigma$.
    Thus it has norm $\ll 1$.
    By Stirling's formula we have
    \begin{equation}
        \frac{\Gamma(s \pm \frac k2)\Gamma(s\mp \frac k2 + 2)}{\Gamma(2s-1)} \ll |t|^{1/2},
    \end{equation}
    so from Lemmas~\ref{lem:poincare-inner-prod-opp-sign} and \ref{lem:poincare-inner-prod-same-sign} and the Cauchy-Schwarz inequality we obtain the bound
    \begin{equation} \label{eq:zmns-bound-poincare}
        Z_{m,n}(s) \ll \lVert P_{\alpha,m}(\tau,s,k) \rVert |t|^{1/2}.
    \end{equation}
    The remainder of the proof follows the proof of Lemma~1 of \cite{GS}. 

    By Corollary~5.7 of \cite{HS} we have
    \[
        \lVert R_{s(1-s)} \rVert \leq \frac{1}{|\im(s(1-s))|} = \frac{1}{|t|(2\sigma-1)},
    \]
    where $R_\lambda = (\Delta_k+\lambda)^{-1}$ is the resolvent.
    By \eqref{eq:poincare-recurrence} we have  $\norm{P_{\alpha,m}(\tau,s,k)} \ll (\sigma-\frac 12)^{-1}$ and together with \eqref{eq:zmns-bound-poincare}, this completes the proof.
\end{proof}

\begin{proof}[Proof of Theorem~\ref{thm:gs}]
    It suffices to prove the theorem for $x\in \Z+\frac 12$.
    Let $T = x^{\delta}$ with $\delta \in (0, 1)$ to be chosen later and let $y$ be a positive real number with $y\neq 1$.
    From Chapter~17 of \cite{davenport} we get
    \begin{equation}
        \frac 1{2\pi i} \int_{a-iT}^{a+iT} \frac{y^s}{s} \, ds = 
        \delta_{y>1}
        + O\left(y^a \min\left(1, \frac{1}{T|\log y|}\right)\right),
    \end{equation}
    where $a$ is any positive number.
    We take $a = B+\ep$. It follows that
    \begin{equation}
        \tfrac 12\sum_{0<|c|\leq x} \frac{S_{\alpha\beta}(m,n,c)}{|c|} \\= \frac 1{2\pi i} \int_{a-iT}^{a+iT} Z_{m,n}\pfrac{s+1}{2} \frac{x^s}{s} \, ds + O\left(\frac{x^{B+\ep}}{T}\sum_{c=1}^\infty \frac{1}{c^{1+\ep}|\log (x/c)|} \right).
    \end{equation}
    In the ranges $c\leq \frac 34x$ and $c\geq \frac 54x$ we have $|\log(x/c)|\gg 1$, so the contribution from those terms is $O(x^{B+\ep}/T)$.
    For the remaining terms with $|\frac cx-1|\leq \frac 14$ we have $\log(x/c) \gg \frac{|c-x|}{x}$, and we find that the contribution from those terms is only larger by a factor of $\log x$.
    As in the proof of the prime number theorem, we shift the contour of integration to $\re(s) = \ep$.
    By \eqref{eq:Z-bound} and the Phragm\'en-Lindel\"of principle, we have
    \begin{equation}
        Z_{m,n}(s) \ll |t|^{\frac 12 - \frac{\sigma}{2B}+\ep}
    \end{equation}
    for all $s$ along the box with corners $\ep-iT$ and $a+iT$.
    Therefore
    \begin{equation}
        \tfrac 12\sum_{0<|c|\leq x} \frac{S_{\alpha\beta}(m,n,c)}{c} = 2\sum_{\frac 12 < s_j < 1} \res_{s=s_j} Z_{m,n}(s) \frac{x^{2s_j-1}}{2s_j-1} + O\left(\frac{x^{B+\ep}}{T} + T^{\frac 12+\ep}x^{\ep}\right).
    \end{equation}
    We take $T = x^{2B/3}$ to balance the error terms.
    If we set
    \begin{equation} \label{eq:aj-formula}
        a_j = 4^{2-s_j}\pi^{1-2s_j} (\tilde m|\tilde n|)^{1-s_j} \frac{\Gamma(s_j+\sgn(\tilde n) \frac k2)\Gamma(2s_j-1)}{(2s_j-1)\Gamma(s_j - \frac k2)} \sum_{u_\ell\in \mathcal C(\lambda_j)} \bar c_{\alpha \ell}(\tilde m) c_{\beta \ell}(\tilde n)
    \end{equation}
    then the theorem follows from \eqref{eq:residue-tau}.
\end{proof}

\appendix
\section{Completions of $\phi$, $\psi$, and $\chi$.}

\label{appendix}
Here we work out the completions of the third order mock theta functions $\phi$, $\psi$, $\nu$, $\chi$, $\sigma$, $\xi$ and $\rho$, and we prove Lemmas~\ref{lem:order-3-2} and \ref{lem:order-3-3}.
The information necessary to prove these transformations is given in \cite{Gordon-McIntosh}, and we follow the general procedure developed by Zwegers in \cite{zwegers-f} (see also \cite{zwegers-thesis}).
Let
\[
    {V_{N,a}}(\tau) := \frac{i^{3/2}}{\sqrt{2N}} \int_{0}^{i\infty} \frac{\theta_{N,a}(z)}{\sqrt{-i(z\tau-1)}}dz.
\]

We begin with the mock theta functions $\phi$, $\psi$ and $\nu$, where
\[\nu(q) = \sum_{n=0}^\infty \frac{q^{n(n+1)}}{(-q;q^2)_{n+1}}.\]
We claim that these three mock theta functions satisfy the following transformations. 
Let 
\[F_{(3,2)}(\tau) = \left(\begin{array}{ccc}\sqrt{8}q^{-1/24}\psi(\tau)\\2q^{1/3}\nu(\tau+\frac12)
\\q^{-1/96}\phi\pfrac{\tau}{4}\\q^{-1/96}\phi\pfrac{\tau+1}{4}
\\q^{-1/96}\phi\pfrac{\tau+2}{4}\\q^{-1/96}\phi\pfrac{\tau+3}{4}\end{array}\right).\]
Then $F_{(3,2)}$ transforms by 
\[F_{(3,2)}(\tau+1)=M_TF_{(3,2)}(\tau),\qquad F_{(3,2)}(-1/\tau) = \sqrt{-i\tau}(M_SF_{(3,2)}+R_{(3,2)})(\tau),\]
where 
\[
M_T = \left(\begin{smallmatrix}
\zeta_{24}^{-1} & 0 & 0 & 0 & 0 & 0 \\
0 & \zeta_3 & 0 & 0 & 0 & 0 \\
0 & 0 & 0 & \zeta_{96}^{-1} & 0 & 0 \\
0 & 0 & 0 & 0 & \zeta_{96}^{-1} & 0 \\
0 & 0 & 0 & 0 & 0 & \zeta_{96}^{-1} \\
0 & 0 & \zeta_{96}^{-1} & 0 & 0 & 0 \\
\end{smallmatrix}\right),\quad
M_S = \left(\begin{smallmatrix}
0 & 0 & 1\vphantom{\zeta_{24}^{-1}} & 0 & 0 & 0 \\
0 & 0 & 0 & 0 & 1\vphantom{\zeta_{24}^{-1}} & 0 \\
1\vphantom{\zeta_{24}^{-1}} & 0 & 0 & 0 & 0 & 0 \\
0 & 0 & 0 & 0 & 0 & \zeta_{16}^3 \\
0 & 1\vphantom{\zeta_{24}^{-1}} & 0 & 0 & 0 & 0 \\
0 & 0 & 0 & \zeta_{16}^{-3} & 0 & 0 \\
\end{smallmatrix}\right),
\]
and
\[R_{(3,2)} = \begin{pmatrix}
\sqrt{2}(V_{24,2}+V_{24,10}+V_{24,14}+V_{24,22})\\
2(V_{24,8}+V_{24,16})\\
-(V_{24,1}+V_{24,5}+V_{24,7}+V_{24,11}-V_{24,13}-V_{24,17}-V_{24,19}-V_{24,23})\\
-(V_{24,1}-iV_{24,5}-V_{24,7}-iV_{24,11}-iV_{24,13}-V_{24,17}-iV_{24,19}+V_{24,23})\\
-(V_{24,1}-V_{24,5}+V_{24,7}-V_{24,11}+V_{24,13}-V_{24,17}+V_{24,19}-V_{24,23})\\
-(V_{24,1}+iV_{24,5}-V_{24,7}+iV_{24,11}+iV_{24,13}-V_{24,17}+iV_{24,19}+V_{24,23})
\end{pmatrix}.\]
The transformation under $\tau\mapsto \tau+1$ is clear. 
We briefly explain how to prove the $\tau\mapsto -1/\tau$ transformation here; a full proof is given in \cite{anderson-thesis}.

The first, second, third, and fifth components of the transformation under $S=\ptmatrix0{-1}10$ follow directly from the identities listed in Section~4 of \cite{Gordon-McIntosh}, possibly with a simple change of variable, alongside the identities
\begin{align}\label{eq:W-klein}
    \frac{i}{\tau}W\pfrac{\pi i}{\tau}&=\frac{1}{\sqrt{6}}(-V_{12,1}-V_{12,5}-V_{12,7}-V_{12,11})(\tau),\\
    \frac{i}{\tau}W_1\pfrac{\pi i}{\tau}&=\frac{1}{\sqrt{6}}(-V_{12,1}+V_{12,5}-V_{12,7}+V_{12,11})(\tau),\\
    \frac{i}{\tau}W_2\pfrac{\pi i}{\tau}&=\frac{1}{\sqrt{6}}(-V_{12,4}-V_{12,8})(\tau).
\end{align}
derived from Lemmas~5.2.2 and 5.2.3 in \cite{klein-thesis} and basic properties of the unary theta functions. 
The remaining two transformations are not as immediate, but follow from the expansion
\[\ptmatrix{1}{1}{0}{4}S\ptmatrix{1}{-1}{0}{4}^{-1}\tau=\ptmatrix{1}{1}{0}{2}S\ptmatrix{1}{-1}{0}{2}^{-1}\ptmatrix{1}{0}{0}{2}S\ptmatrix{1}{1}{0}{2}^{-1}\ptmatrix{1}{0}{0}{2}S\ptmatrix{1}{0}{0}{2}^{-1}\tau,\]
alongside three of the transformations from Section 4 of \cite{Gordon-McIntosh}. The process of simplifying is tedious but straightforward using the modularity properties of the unary theta functions. 

We follow Zwegers to form the completions. 
Define $G_{(3,2)}$ in the same way as $R_{(3,2)}$ but with $V_{24,h}$ replaced by $\V_{24,h}$. 
Let $\hat H_{(3,2)}=F_{(3,2)}-G_{(3,2)}.$
Then $\hat H_{(3,2)}$ transforms as
\begin{equation} \label{eq:H-32-transform}
\hat H_{(3,2)}(\tau+1) = M_T \hat H_{(3,2)}(\tau),\quad \hat H_{(3,2)}(-1/\tau) = \sqrt{-i\tau} M_S \hat H_{(3,2)}(\tau).
\end{equation}
If we let $2\tilde\nu$ denote the second component of $\hat H_{(3,2)}$ then Lemma~\ref{lem:order-3-2} can be easily verified by a computer algebra system using \eqref{eq:H-32-transform}.

The process for constructing the vector-valued harmonic Maass form encoding $\chi$, $\rho$, $\xi$, and $\sigma$ is similar, and the details can again be found in \cite{anderson-thesis}. 
Define
\begin{align}
\rho(q) &= \sum_{n=0}^\infty \frac{q^{2n(n+1)}(q;q^2)_{n+1}}{(q^3;q^6)_{n+1}}, \\
\xi(q) &= 1+2q\sum_{n=1}^\infty \frac{q^{6n(n-1)}}{(q;q^6)_n(q^5;q^6)_n},
\\
\sigma(q) &= \sum_{n=1}^\infty \frac{q^{3n(n-1)}}{(-q;q^3)_n(-q^2;q^3)_n}.
\end{align}
We define $\hat H_{(3,3)} = F_{(3,3)} - G_{(3,3)}$, where
\begin{multline}
    F_{(3,3)}(\tau) = 
    \Big(2\sqrt{3}\,q^{7/8}\sigma(\tau) , \,
\sqrt{3}\,\xi(\tfrac {\tau}{2}) , \,
\sqrt{3}\,\xi(\tfrac {\tau+1}{2}) , \,
2q^{-1/72}\chi(\tfrac {\tau}{3}) , \,
2q^{-1/72}\chi(\tfrac {\tau+1}{3}) , \,
2q^{-1/72}\chi(\tfrac {\tau+2}{3}) , \\
2q^{1/9}\rho(\tfrac {\tau}{6}), \,
2q^{1/9}\rho(\tfrac {\tau+1}{6}) , \,
2q^{1/9}\rho(\tfrac {\tau+2}{6}), \,
2q^{1/9}\rho(\tfrac {\tau+3}{6}), \,
2q^{1/9}\rho(\tfrac {\tau+4}{6}), \,
2q^{1/9}\rho(\tfrac {\tau+5}{6})\Big)^T
\end{multline}
and
\[G_{(3,3)}=\begin{pmatrix}
\sqrt{3}(-\V_{18,3}+ \V_{18,15})\\
\sqrt{3}(\V_{18,6}+ \V_{18,12}) \\
\sqrt{3}(-\V_{18,6}+ \V_{18,12}) \\
-(\V_{18,1}-\V_{18,5}+\V_{18,7}-\V_{18,11}+\V_{18,13}-\V_{18,17}) \\
-(\V_{18,1}-\zeta_3^{-1}\V_{18,5}+\zeta_3\V_{18,7}-\zeta_3\V_{18,11}+\zeta_3^{-1}\V_{18,13}-\V_{18,17}) \\
-(\V_{18,1}-\zeta_3\V_{18,5}+\zeta_3^{-1}\V_{18,7}-\zeta_3^{-1}\V_{18,11}+\zeta_3\V_{18,13}-\V_{18,17}) \\
-(\V_{18,2}+\V_{18,4}-\V_{18,8}-\V_{18,10}+\V_{18,14}+\V_{18,16}) \\
-(\zeta_6^{-1}\V_{18,2}+\zeta_3^{-1}\V_{18,4}-\V_{18,8}+\V_{18,10}+\zeta_6\V_{18,14}+\zeta_3\V_{18,16}) \\
-(\zeta_3^{-1}\V_{18,2}+\zeta_3\V_{18,4}-\V_{18,8}-\V_{18,10}+\zeta_3\V_{18,14}+\zeta_3^{-1}\V_{18,16}) \\
-(-\V_{18,2}+\V_{18,4}-\V_{18,8}+\V_{18,10}-\V_{18,14}+\V_{18,16}) \\
-(\zeta_3\V_{18,2}+\zeta_3^{-1}\V_{18,4}-\V_{18,8}-\V_{18,10}+\zeta_3^{-1}\V_{18,14}+\zeta_3\V_{18,16})  \\
-(\zeta_6\V_{18,2}+\zeta_3\V_{18,4}-\V_{18,8}+\V_{18,10}+\zeta_6^{-1}\V_{18,14}+\zeta_3^{-1}\V_{18,16}) 
\end{pmatrix}.\]
Then
\[\hat H_{(3,3)}(\tau+1) = M_T \hat H_{(3,3)}(\tau),\quad \hat H_{(3,3)}(-1/\tau) = \sqrt{-i\tau}M_S \hat 
 H_{(3,3)}(\tau),\]
where 
\[M_T=\lrp{\begin{smallmatrix}
\zeta_8^7 & 0 & 0 & 0 & 0 & 0 & 0 & 0 & 0 & 0 & 0 & 0  \\
0 & 0 & 1 & 0 & 0 & 0 & 0 & 0 & 0 & 0 & 0 & 0  \\
0 & 1 & 0 & 0 & 0 & 0 & 0 & 0 & 0 & 0 & 0 & 0  \\
0 & 0 & 0 & 0 & \zeta_{72}^{-1} & 0 & 0 & 0 & 0 & 0 & 0 & 0  \\
0 & 0 & 0 & 0 & 0 & \zeta_{72}^{-1} & 0 & 0 & 0 & 0 & 0 & 0  \\
0 & 0 & 0 & \zeta_{72}^{-1} & 0 & 0 & 0 & 0 & 0 & 0 & 0 & 0  \\
0 & 0 & 0 & 0 & 0 & 0 & 0 & \zeta_9 & 0 & 0 & 0 & 0  \\
0 & 0 & 0 & 0 & 0 & 0 & 0 & 0 & \zeta_9 & 0 & 0 & 0  \\
0 & 0 & 0 & 0 & 0 & 0 & 0 & 0 & 0 & \zeta_9 & 0 & 0  \\
0 & 0 & 0 & 0 & 0 & 0 & 0 & 0 & 0 & 0 & \zeta_9 & 0  \\
0 & 0 & 0 & 0 & 0 & 0 & 0 & 0 & 0 & 0 & 0 & \zeta_9  \\
0 & 0 & 0 & 0 & 0 & 0 & \zeta_9 & 0 & 0 & 0 & 0 & 0 
\end{smallmatrix}},\quad M_S=\left(\begin{smallmatrix}
0 & 0 & 0 & 0 & 0 & 0 & -1 & 0 & 0 & 0 & 0 & 0  \\
0 & 0 & 0 & 1 & 0 & 0 & 0 & 0 & 0 & 0 & 0 & 0  \\
0 & 0 & 0 & 0 & 0 & 0 & 0 & 0 & 0 & 1 & 0 & 0  \\
0 & 1 & 0 & 0 & 0 & 0 & 0 & 0 & 0 & 0 & 0 & 0  \\
0 & 0 & 0 & 0 & 0 & 0 & 0 & 0 & \zeta_{36} & 0 & 0 & 0  \\
0 & 0 & 0 & 0 & 0 & 0 & 0 & 0 & 0 & 0 & \zeta_{36}^{-1} & 0  \\
-1 & 0 & 0 & 0 & 0 & 0 & 0 & 0 & 0 & 0 & 0 & 0  \\
0 & 0 & 0 & 0 & 0 & 0 & 0 & 0 & 0 & 0 & 0 & \zeta_{36}  \\
0 & 0 & 0 & 0 & \zeta_{36}^{-1} & 0 & 0 & 0 & 0 & 0 & 0 & 0  \\
0 & 0 & 1 & 0 & 0 & 0 & 0 & 0 & 0 & 0 & 0 & 0  \\
0 & 0 & 0 & 0 & 0 & \zeta_{36} & 0 & 0 & 0 & 0 & 0 & 0  \\
0 & 0 & 0 & 0 & 0 & 0 & 0 & \zeta_{36}^{-1} & 0 & 0 & 0 & 0 
\end{smallmatrix}\right).\]
Finally, we write $\hat H_{(3,3)} = (H_1, \ldots ,H_{12})^T$ and set 
\begin{alignat*}{2}
\tilde{\sigma} &= \frac{1}{2\sqrt{3}}H_{1}
&\tilde{\xi}^{(k)} &= \frac{1}{2\sqrt{3}}\sum_{\ell=0}^1 (-1)^{k \ell} H_{\ell+2}\\
\tilde{\chi}^{(k)} &= \frac{1}{6}\sum_{\ell=0}^2 \zeta_3^{-k\ell} H_{\ell+4} \qquad 
&\tilde{\rho}^{(k)} &= \frac{1}{12}\sum_{\ell=0}^5 \zeta_6^{-k\ell} H_{\ell+7}.
\end{alignat*}
The holomorphic parts of $\tilde{\xi}^{(k)},\tilde{\chi}^{(k)},$ and $\tilde{\rho}^{(k)}$ encode the coefficients of $\xi,\chi,\rho$ of index  $n\equiv k$ modulo $2,3$ and $6$, respectively. 
Lemma~\ref{lem:order-3-3} now follows from a computer algebra verification.


\bibliographystyle{plain}
\bibliography{bibliography}
\end{document}